\documentclass[12pt]{amsart}

\usepackage{amsmath,amssymb,mathdots}
\usepackage{fancybox,enumerate}

\newcommand{\nc}{\newcommand}

\newcommand{\beq}{\begin{equation}}
\newcommand{\eeq}{\end{equation}}
\newcommand{\barr}{\begin{array}}
\newcommand{\earr}{\end{array}}
\newcommand{\beqar}{\begin{eqnarray}}
\newcommand{\eeqar}{\end{eqnarray}}
\newtheorem{theorem}{Theorem}[section]
\newtheorem{corollary}[theorem]{Corollary}
\newtheorem{lemma}[theorem]{Lemma}
\newtheorem{prop}[theorem]{Proposition}
\theoremstyle{definition}
\newtheorem{definition}[theorem]{Definition}
\newtheorem{conjecture}[theorem]{Conjecture}
\theoremstyle{remark}
\newtheorem{remit}[theorem]{Remark}
\newtheorem{exit}[theorem]{Example}

\newenvironment{rem}{\begin{remit}\rm}{\end{remit}}
\newenvironment{ex}{\begin{exit}\rm}{\end{exit}}

\newcommand{\diag}{{\rm diag}}

\newcommand{\CC}{{\mathbb C }}

\newcommand{\ZZ}{{\mathbb Z }}
\newcommand{\PP}{ {\mathbb P } }
\newcommand{\QQ}{{\mathbb Q }}

\newcommand{\cala}{\mathcal{A}}
\newcommand{\calb}{\mathcal{B}}

\newcommand{\calq}{\mathcal{Q}}

\newcommand{\cals}{\mathcal{S}}

\newcommand{\calv}{\mathcal{V}}

\newcommand{\calw}{\mathcal{W}}

\newcommand{\lieg}{{\mathfrak g}}
\newcommand{\lieh}{{\mathfrak h}}

\newcommand{\liet}{{\mathfrak t}}

\nc{\lieq}{{\mathfrak q}}
\nc{\liez}{{\mathfrak z}}
\nc{\lieqs}{{\lieq}^*}
\nc{\liegs}{{\lieg}^*}
\nc{\liep}{{\mathfrak p}}
\nc{\lieps}{{\liep}^*}

\newcommand{\im}{ \,{\rm Im} \,}

\newcommand{\GL}{\mathrm{GL}}
\newcommand{\SL}{\mathrm{SL}}
\newcommand{\Sp}{\mathrm{Sp}}

\newcommand{\SO}{\mathrm{SO}}

\newcommand{\so}{\mathfrak{so}}
\newcommand{\spc}{\mathfrak{sp}}
\newcommand{\slc}{\mathfrak{sl}}

\def\a{\alpha}

\def\b{\beta}
\def\g{\gamma}

\def\t{\theta}

\def\l{\lambda}

\newcommand{\grass}{\mathrm{Gr}}

\newcommand{\tS}{\tilde{S}}
\newcommand{\tF}{\tilde{F}}
\nc{\umax}{{U_{\max}}}

\newcommand{\bu}{\mathbf{u}}

\newcommand{\bj}{\mathbf{j}}
\newcommand{\bi}{\mathbf{i}}
\newcommand{\bv}{\mathbf{v}}
\newcommand{\bw}{\mathbf{w}}
\newcommand{\br}{\mathbf{r}}

\newcommand{\te}{\tilde{e}}

\title{On the Popov-Pommerening conjecture for linear algebraic groups}

\author{Gergely B\'erczi}
\email{berczi@maths.ox.ac.uk}
\address{Mathematical Institute, University of Oxford, Andrew Wiles Building, OX2 6GG Oxford, UK} 
\thanks{This work was partially supported by the Engineering and Physical Sciences 
Research Council [grant numbers   GR/T016170/1,EP/G000174/1].}

\begin{document}

\begin{abstract}
Let $G$ be a reductive group over an algebraically closed subfield $k$ of $\CC$ of characteristic zero, $H \subseteq G$ an observable subgroup normalized by a maximal torus of $G$ and $X$ an affine $k$-variety acted on by $G$. Popov and Pommerening conjectured in the late 70's that the invariant algebra $k[X]^H$ is finitely generated. We prove the conjecture for 1) subgroups of $\SL_n(k)$ closed under left (or right) Borel action and for 2) a class of Borel regular subgroups of classical groups. We give a partial affirmative answer to the conjecture for general regular subgroups of $\SL_n(k)$.
\end{abstract}

\maketitle
\section{Introduction}

Let $k$ be an algebraically closed subfield of $\CC$ of characteristic zero, $G$ an affine algebraic group over $k$,
and $X$ an affine $k$-variety on which $G$ acts rationally. There is an induced
action of $G$ on the coordinate ring $k[X]$ given by $(g\cdot f)(x) = f(g^{-1}x )$ for $g\in G, f\in k[X]$ and $x\in X$. The invariant subalgebra is $k[X]^G= \{f\in  k[X]|g\cdot f =f \text{ for all } g\in G\}$. Invariant theory studies the structure of this algebra and one of its fundamental problems is to characterise those actions where $k[X]^G$ is finitely generated. 

When $G$ is reductive $k[X]^G$ is finitely generated, due to Mumford \cite{mumford},
Nagata \cite{nagata} and Haboush \cite{haboush}. Since Nagata's counterexample from 1958 \cite{nagata} we know that for non-reductive groups the invariant algebra is not necessarily finitely generated. In fact Popov \cite{Popov} proved that finite generation for arbitrary ring $k[X]^G$ implies that $G$ is reductive. 

Invariant rings for non-reductive group actions have been extensively studied over the last 60 years. Finite generation has been proved in many interesting situations; however, characterisation of those actions with finitely generated invariant rings is still seems to be hopeless. 
Weitzenb\"ock in \cite{Weitzenbock} (and later Seshadri \cite{Seshadri}) proved that finite generation holds if $G$ is the additive group $k^+$ of an algebraically closed field $k$ of characteristic $0$, $X$ is an
affine $k$-space and the action of $k^+$ on $X$ extends to $\SL(2,k)$. Later Hochschild, Mostow and
Grosshans generalised this result by showing that if $G$ is reductive, $H$ is the unipotent radical of some
parabolic subgroup of $G$ and $G$ acts rationally on $X$, then $k[X]^H$ is finitely
generated (see \cite{hm,Grosshans2}). A natural generalisation of parabolic subgroups of the reductive group $G$ are the subgroups normalized by a maximal torus of $G$: these are generated by root subgroups corresponding to a closed set of positive roots and we call them regular subgroups. The following conjecture was formulated almost simultaneously in the late 1970's by Popov and Pommerening. 

\begin{conjecture}[(Popov,  Pommerening)]\label{conjecture1} Let $G$ be a reductive group over $k$, and let $H \subset G$ be an observable subgroup normalized by a maximal torus of $G$. Then for any affine $G$-variety $X$ the algebra of invariants $k[X]^H$ is finitely generated.
\end{conjecture}

Note that $H=U \rtimes R$ can be written as the semidirect product of its unipotent radical $U$ and a reductive group $R$. Since $U\le H$ is a characteristic subgroup, if the maximal torus normalizes $H$ then it normalizes $U$ as well. Moreover, 
\[k[X]^H=[k[X]^U]^R\]
holds for any $H$-variety $X$, so we can restrict our study to the unipotent part $U$ and the corresponding invariant ring $k[X]^U$. Furthermore, if $G$ is any linear algebraic group and $H \subset G$ is a closed subgroup and $X$ is an affine $G$-variety, then 
\[k[X]^H=(k[G]^{H}\otimes k[X])^{G}\]
holds for the invariant rings. This is called the transfer principle, which goes back to the nineteenth century. In its modern form, it appeared in Grosshans \cite{Grosshans3} and Popov \cite{Popov3}. In particular, if $G$ is reductive, then finite generation of $k[G]^U$ implies finite generation of $k[X]^{U}$ and $k[X]^{H}$. 

The question can be further reduced to connected, simply connected, simple reductive groups, see \cite{Tan3, Grosshans}. Unipotent subgroups of these normalized by a maximal torus $T$ can be parametrized by '(quasi)closed' subsets $S \subset R^+$ of the set $R^+$ of positive roots of $G$ relative to some Borel subgroup containing $T$. If $\mathrm{char}(k)=0$ then quasi-closed and closed subsets are the same: a subset $S\subset R^+$ is closed if the subgroup $\langle U_\a:\a \in S\rangle$ generated by the root subgroups in $S$ does not contain any $U_\b$ with $\b \in R \setminus S$. Then the unipotent group 
\[U_S=\langle U_{\a}:\a \in S\rangle \]
is normalized by the maximal torus $T$, and all unipotent subgroups of $G$ with this property have this form for some closed subset $S$ (cf. \cite{Tan2,Tan3}).

When $G=\SL_n(k)$ we can assume that $U_S$ is unipotent upper triangular subgroup normalized by the diagonal maximal torus. Let $B_n\subset \SL_n(k)$ be the upper Borel then the positive roots are $R^+=\{\a_i-\a_j:i<j\}$, where $\a_i:T \to k^*$ is the character of the maximal diagonal torus $T\subset \SL_n(k)$ sending a torus element to its $(i,i)$ entry. Then $S\subset R^+$ is closed if and only if it is the incidence matrix of a strict ordering of the set $\{1,\ldots, n\}$, that is, 
\begin{equation}\label{partialorder1}
(i,j),(j,k) \in S \Rightarrow (i,k)\in S.
\end{equation}
The corresponding unipotent subgroup $U_S$ is upper triangular with $1$'s on the diagonal and zeros at the entries $(i,j)$ where $\a_i-\a_j \notin S$. For example, for $S=\{\a_1-\a_3,\a_2-\a_4\}$
\begin{equation}\label{regularsubgroup}
U_{S}=\begin{array}{|cccc|}
\hline 
1 & 0 & \bullet & 0 \\
  & 1 & 0 & \bullet \\
  &   & 1 & 0 \\
  &   &   & 1 \\
\hline
\end{array}=\left\{\left(\begin{array}{cccc}
1 & 0 & a & 0\\
  & 1 & 0 & b\\
  &   & 1 & 0\\
  &   &   & 1
 \end{array}\right):a,b\in k \right\}\subset \SL_4(k).
\end{equation} 
We will often refer to elements of $S$ with pairs $(i,j)$ instead of $\a_i-\a_j$. We introduce the following special type of regular subgroups.  
\begin{enumerate}
\item A regular subgroup $U_S \subset B_n$ is called \textbf{left (resp right) Borel-regular} if its Lie algebra $\mathfrak{u}_S \subset \slc_n$ is closed under multiplication on the left (resp. right) by $B_n$. 
In a left Borel-regular subgroup the nondiagonal elements of $S$ form a vertical 'barcode', that is $(i,j)\in S \Rightarrow (i-1,j)\in S$ holds and $S$ can be parametrized by the sequence of positions of the lowest free parameter in each column:
\begin{equation}\nonumber
 U^{0,0,2,1,4,2}=\begin{array}{|cccccc|}
\hline 
1 & 0 & \bullet & \bullet & \bullet & \bullet \\
  & 1 & \bullet & 0 & \bullet & \bullet \\
  &   & 1 & 0 & \bullet  & 0 \\
  &   &   & 1 & \bullet & 0 \\
  & & & & 1 & 0 \\
  & & & & & 1\\
\hline
\end{array}
\end{equation}

\item We call a subgroup $U_S \subset B_n$ \textbf{Borel-regular} if it is left and right Borel-regular at the same time. Equivalently, it is normalized by $B_n$. This means that $(i,j)\in S \Rightarrow (i,j+1),(i-1,j)\in S$, hence Borel-regular subgroups are those left Borel-regular subgroups which correspond to some increasing sequence, e.g. 
\begin{equation}\nonumber U^{0,0,1,1,3,4}=\begin{array}{|cccccc|}
\hline 
1 & 0 & \bullet & \bullet & \bullet & \bullet \\
  & 1 & 0 & 0 & \bullet & \bullet \\
  &   & 1 & 0 & \bullet & \bullet \\
  &   &   & 1 & 0 & \bullet \\
  & & & & 1 & 0 \\
  & & & & & 1\\
\hline
\end{array}
\end{equation}
Unipotent radicals of parabolic subgroups are special Borel-regular subgroups where all blocks "touch" the main diagonal.    

Borel-regular subgroups can be defined in any linear algebraic group $G$: these are the regular subgroups normalized by some Borel subgroup of $G$. These are subgroups of the form $U_S \subset G$ where $S\subset R^+$ is closed under shifting by elements of $R^+$, i.e. $S+r \subseteq S$ for any $r\in R^+$. 

\item In particular, the symplectic group $\Sp_{n}(k)$ and orthogonal group $\SO_n(k)$ have Borel-compatible embeddings into $\SL_n(k)$, that is, a choice of Borel subgroups $B_{\Sp_n}$ and $B_{\SO_n}$ whose image sit in the upper Borel $B_n$ of $\SL_n(k)$. The image of the maximal torus in $\Sp_n(k)$ and $ \SO_n(k)$ consists of diagonal matrices $\mathrm{diag}(t_1,\ldots, t_n)$ satisfying $t_i=t_{n+1-i}^{-1}$, see \S\ref{subsec:symplectic}, \S\ref{subsec:orthogonal} for the details. The positive roots are 
\[R^+=\begin{cases} \{\a_i-\a_j\}_{1\le i<j \le l}\cup \{\a_i+\a_j\}_{1\le i\le j\le l} & \text{ for } \Sp_{2l}\\
 \{\a_i-\a_j\}_{1\le i<j \le l}\cup \{\a_i+\a_j\}_{1\le i< j\le l} & \text{ for } \SO_{2l} \\
 \{\a_i-\a_j\}_{1\le i<j\le l}\cup \{\a_i+\a_j\}_{1\le i< j\le l} \cup \{\a_i\}_{1\le i \le l} & \text{ for } \SO_{2l+1}.
\end{cases}\]

We call the Borel-regular subgroup $U_S \subset \Sp_n(k)$ \textbf{fat Borel-regular} if $\{\a_i+\a_j:1\le i\le j\le l\} \subseteq S$. Similarly, $U_S \subset \SO_n(k)$ is fat if $\{\a_i+\a_j:1\le i< j\le l\} \subseteq S$ when $n=2l$ and $\{\a_i+\a_j,\a_i:1\le i< j\le l\} \subseteq S$ when $n=2l+1$.  
We will see that if $S^{\SL} \subset \{1,\ldots, n\}^2$ collects the possible nonzero entries of a fat Borel-regular subgroup $U_S \subset \Sp_n(k),\SO_n(k) \subset \SL_n(k)$ then $U_S$ is Borel regular in $\SL_n(k)$ such that $S$ is a 'fat' domain in the sense that it contains the top right quarter of $\SL_n$: 
\begin{equation}\nonumber U^{0,0,1,3,3,4}=\begin{array}{|cccccc|}
\hline 
1 & 0 & \bullet & \bullet & \bullet & \bullet \\
  & 1 & 0 & \bullet & \bullet & \bullet \\
  &   & 1 & \bullet & \bullet & \bullet \\
  &   &   & 1 & 0 & \bullet \\
  & & & & 1 & 0 \\
  & & & & & 1\\
\hline
\end{array}
\end{equation}

\end{enumerate}

Machinery for proving finite generation for algebra of invariants is quite limited. However, there exists a standard criterion, called the Grosshans criterion \cite{Grosshans3,Grosshans} for proving the finite generation of $k[G]^H$, where $H\subset G$ is observable in the sense that 
$$H = \{ g \in G : f(xg)=f(x) \mbox{ for all $x \in G$ and } f \in k[G]^H \}.$$
Note that the action of $H$ on $G$ is by right translation. In this case the finite generation of $k[G]^H$ is equivalent to the existence of a finite-dimensional affine (left) $G$-module $\calw$ and some $w \in \calw$ such that $H=G_w$ is the stabiliser of $w$ and $\dim(\overline{G\cdot w}\setminus G\cdot w)\le \dim(G\cdot w)-2$.  Such subgroups $H$ are called Grosshans subgroups of $G$ and we call the pair $(\calw,w)$ a Grosshans pair for $H$. 

The main results of this paper are the following two theorems 

\begin{theorem}\label{main}
If $U_S\subset \SL_n(k)$ is a left (resp. right) Borel-regular subgroup, then $U_S$ is a Grosshans subgroup of $\SL_n(k)$. Therefore every linear action of $U_S$ on an affine or projective variety which extends to a linear action of $G$ has a finitely generated algebra of invariants. In particular this gives an affirmative answer to the Popov-Pommerening conjecture for left (resp. right) Borel-regular subgroups of $\SL_n(k)$.
\end{theorem}

\begin{theorem}\label{main2}
Let $G$ be a linear reductive group over $k$ of type $B$ or $D$ and $U_S\subset G$ a fat Borel-regular subgroup. Then $U_S$ is a Grosshans subgroup of $G$. In particular this gives an affirmative answer to the Popov-Pommerening conjecture for fat Borel-regular subgroups of symplectic and orthogonal Lie groups.
\end{theorem}

The Popov-Pommerening conjecture was known before in a few special cases. In a series of papers Tan \cite{Tan,Tan2,Tan3} proved it for all simple groups of Dynkin type $A_n$ with $n\le 4$, and for groups of type $B_2$ and $G_2$. Grosshans in \cite{Grosshans} confirmed the conjecture for unipotent radicals of parabolic subgroups and in \cite{Grosshans4} for those $S\subset R^+$ where $R^+ \setminus S$ is a linearly independent set over $\QQ$. Pommerening \cite{pomm2} proved the conjecture for a large class of subgroups of $\mathrm{GL}_n(k)$ by giving a generating set of the invariant ring $k[\mathrm{GL}_n(k)]^{U_S}$, but these cases only cover very special block regular subgroups. For more details on the history of the problem see \cite{Grosshans} and the survey papers \cite{Tan3,pomm3,Grosshans0}. After finishing the first version of this paper, V. Popov kindly drew my attention to the unpublished PhD Thesis of A'Campo-Neuen \cite{campo-neuen} where the Popov-Pomerening conjecture is proved for Borel-regular subgroups of $\SL_n(k)$, this is a special case of our Theorem \ref{main2}.

The layout of this paper is the following. We will work with $k=\CC$, but all arguments work for any algebraically closed field $k$ of characteristic zero which is a subfield of $\CC$. We start with a short introduction of Grosshans subgroups in \S \ref{section:grosshans}. In \S\ref{sec:construction} we construct for any regular subgroup $U_S \subset \SL_n(k)$ corresponding to the closed subset $S \subseteq R^+$ a subset family $\tS \subset 2^{\{1,\ldots n\}}$ and an affine $\SL_n(k)$-module $\calw_{\tS}$ with a point $p_{\tS}\in \calw_{\tS}$ whose stabiliser is isomorphic to $U_S$. 

In \S\ref{blocksln} we prove that the constructed pair $(\calw_{\tS},p_{\tS})$ is a Grosshans pair if $U_S$ is left Borel-regular. An outline of the proof is as follows. When the field of definition is $\CC$, the Zariski-closure of an orbit is the Euclidean closure, see \cite{borel}. Therefore every boundary point in $\overline{\SL_n \cdot p_{\tS}}\setminus \SL_n \cdot p_{\tS}$ can be written as a limit $p^\infty=\lim_{m\to \infty} g^{(m)}\cdot p_{\tS}$ for some sequence $(g^{(m)}) \subset \SL_n$. When $k=\CC$, however, $\overline{\SL_n \cdot p_{\tS}}=\SL_n \cdot \overline{B_n \cdot p_{\tS}}$ holds because $\SL_n \cdot p_{\tS}=\SL_n \times_{B_n} (B_n\cdot  p_{\tS})$ fibres over $\SL_n/B_n=\mathrm{Flag}_n$, the complete flag variety, which is closed. Therefore we can study the boundary of the Borel orbit instead. We construct a cover $\overline{B_n \cdot p_{\tS}}\setminus B_n \cdot p_{\tS} =\cup_{\bu,r} \calb_\bu^r$ 
with Borel-invariant boundary subsets $\calb_\bu^r$ indexed by an array $\bu \subset \{1,\ldots, n\}$ and an integer $1\le r \le n$. We prove that for all $\bu,r$ either (1) $\dim(\calb_\bu^r)\le \dim(B_n \cdot p_{\tS})-2$ or (2) every point of $\calb_\bu^r$ is fixed by a $1$-dimensional subgroup of the opposite Borel $B^{op}$. In both cases we can easily deduce that $\dim(\SL_n \cdot \calb_\bu^r)\le \dim(\SL_n \cdot p_{\tS})-2$.

In \S\ref{sec:classical} we study Borel regular subgroup of classical groups. In particular, in \S\ref{subsec:symplectic} we define a Borel-compatible embedding of $\Sp_n(k)$ into $\SL_n(k)$ and using this embedding we parametrize Borel regular subgroups of $\Sp_n(k)$ with root subsets $S\subset \{1,\ldots, n\}^2$ again. We define the symplectic fundamental domain $F\subset \{1,\ldots, n\}^2$ corresponding to $S$ and define the pair $(\calw_{\tF},p_{\tF})$ where the stabiliser of $p_{\tF}$ in $\Sp_n(k)$ is $U_S$. Finally, we prove Theorem \ref{main2} for symplectic groups. \S \ref{subsec:orthogonal} follows the same line for the orthogonal groups $\SO_n$ for odd and even $n$ and proves Theorem \ref{main2} for orthogonal groups.

We conjecture that $(\calw_{\tS},p_{\tS})$ is a Grosshans pair for arbitrary regular subgroup $U_S \subset \SL_n(k)$, not just for left Borel-regular subgroups. Unfortunately we can not prove this in full generality. What we conjecture is that all boundary components of $\SL_n(k) \cdot p_{\tS}$ have codimension at least two in its closure in $\calw_{\tS}$. In \S\ref{strategy} we prove this for a special class of boundary components. We define the toric closure of the orbit $\SL_n(k) \cdot p_{\tS}$ as $\SL_n(k) \cdot (\overline{T \cdot p_{\tS}})$ and prove that the toric boundary components are small. 
\begin{theorem}[(Partial answer to the PP conjecture for general $S$)]\label{mainc}
Let $U_S\subset \SL_n(k)$ be a regular subgroup corresponding to the closed subset $S \subset R^+$.
Then the toric boundary components of the orbit $\SL_n(k) \cdot p_{\tS}$ have codimension at least $2$ in the orbit closure, that is 
\[\dim(\SL_n(k)\cdot (\overline{T\cdot p_{\tS}})\setminus \SL_n(k)\cdot p_{\tS}) \le \dim(\SL_n(k) \cdot p_{\tS})-2.\]
\end{theorem}

We finish the paper with some remarks in \S7 on the relation of our approach to configuration varieties and Bott-Samelson varieties.

\noindent\textbf{Acknowledgments}
I would like to thank Frances Kirwan for many helpful discussions and for listening the details of this work. I thank to Vladimir Popov for pointing out the existence of \cite{campo-neuen} and to Frank Grosshans and Klaus Pommerening for the useful comments to the first version of this paper.

\section{Grosshans subgroups}\label{section:grosshans}

Let $G$ be reductive algebraic group over an algebraically closed field $k$. 
\begin{definition}
A subgroup $H\subset G$ is called Grosshans subgroup if $k[X]^H$ finitely generated for any affine variety $X$ endowed with a linear action of $G$. 
\end{definition}

\begin{theorem}[(Grosshans Criterion \cite{Grosshans})]\label{thmgrosshans} Let $G$ be a reductive group over an algebraically closed field, and $H$ an observable subgroup, that is, $H=\{g \in G|f(xg)=f(x)\}$ for all $x\in G$ and $f\in k[X]^H$. Then the following conditions are equivalent:
\begin{enumerate}
\item $H$ is a Grosshans subgroup of $G$.
\item $k[G]^H$ is a finitely generated $k$-algebra, where $H$ acts via right translations.
\item There is a finite-dimensional left $G$-module $\calw$ and some $w \in \calw$ such that $H=G_w$ is the stabiliser of $w$ (and therefore $G/H$ is a homogeneous space $G \cdot w$) and $\dim(\overline{G\cdot w}\setminus G\cdot w)\le \dim(G\cdot w)-2$. 
\end{enumerate}
\end{theorem}

\begin{definition} A pair $(\calw,w)$ where $\calw$ is a finite dimensional $G$-module and $w\in \calw$ is a point satisfying the Grosshans criterion is called a Grosshans pair.  
\end{definition}

\section{Construction of Grosshans pairs for $G=\SL_n(k)$}\label{sec:construction}

For the rest of the paper we restrict our attention to the $k=\CC$ case, but all arguments work for any algebraically closed field $k$ of characteristic zero which is a subfield of $\CC$. We will often use the shorthand notation $\SL_n$ for $\SL_n(\CC)$.

In this section we assume that $G=\SL_n(\CC)
$ and let $T\subset \SL_n(\CC)$ be the diagonal maximal torus and $\liet \subset \slc_n(\CC)$ the Cartan subalgebra of diagonal matrices. Let $\a_i \in \liet^*$ be the dual of $E_{ii} \in \liet$ where $E_{ii}$ is the matrix of the endomorphism which fixes the $i$th basis vector and kills all other basis vectors.
Let $R^+=\{\a_i-\a_j:i<j\}$ be the set of positive roots and $B_n \subset \SL_n(\CC)$ be the corresponding upper Borel subgroup. The one dimensional root subgroup $U_{\a_i-\a_j}$ consists of unipotent matrices with the only nonzero off-diagonal entry sitting at $(i,j)$. 

A subset $S\subset R^+$ of the root system is closed if and only if the following transitivity conditions hold for all $1\le i<j<k\le n$:
\begin{equation}\label{transitivity}
\a_i-\a_j,\a_j-\a_k \in S \Rightarrow \a_i-\a_k\in S. 
\end{equation}
Define the unipotent subgroup
\begin{equation}\label{defus}
U_S=\langle U_{\a_i-\a_j}:\a_i-\a_j \in S\rangle \subset \SL_n(\CC)
\end{equation} 
generated by the root subgroups $U_{\a_i-\a_j}$. Then $U_S$ is unipotent with independent parameters at the entries indexed by $S$ and it is normalized by the maximal diagonal torus in $\SL_n(\CC)$, and all unipotent subgroups normalized by this torus have this form.   

For $1\le j\le n$ let $S_j=\{j\}\cup \{i:(i,j)\in S\}\subset \{1,\ldots, n\}$ collect the positions of the (possibly) nonzero entries in the $j$th column of $U_S$. In the example \eqref{regularsubgroup} of the Introduction 
\[S_1=\{1\},S_2=\{2\},S_3=\{1,3\},S_4=\{2,4\}.\]

Fix a basis $\{e_1,\ldots e_n\}$ of $\CC^n$ compatible with $B_n$, that is, $B_n$ preserves the subspace $\mathrm{Span}(e_1,\ldots, e_i)$ for $1\le i \le n$. For a subset $Z\subset \{1,\ldots, n\}$ we set
\[p_Z=\wedge_{z\in Z}e_z \in \wedge^{|Z|}\CC^n\]
where $|Z|$ is the cardinality of $Z$. Define the point
\[p_S=\bigoplus_{j=1}^n p_{S_j}=\bigoplus_{j=1}^n \wedge_{i\in S_j} e_i \in 
\calw_S\]
where 
\[\calw_S=\bigoplus_{j=1}^n \wedge^{|S_j|}\CC^n.\]
In the example \eqref{regularsubgroup} 
\[p_S=e_1 \oplus e_2 \oplus (e_1 \wedge e_3) \oplus (e_2 \wedge e_4) \in \CC^S=\CC^4 \oplus \CC^4 \oplus \wedge^2 \CC^4 \oplus \wedge^2 \CC^4.\]

\begin{theorem}\label{stabiliser}
The stabiliser of $p_S$ in $\SL_n$ is $U_S$. 
\end{theorem}

\begin{proof}
Let $T^{S_j} \subset \SL_n(\CC)$ denote the diagonal torus 
\[T^{S_j}=\{\mathrm{diag}(a_1,\ldots, a_n):\Pi_{i\in S_j}a_i=\Pi_{i\notin S_j}a_i=1\}.\]
The stabiliser of $p_S$ is the intersection of the stabilisers of its direct summands. The stabiliser in $\SL_n(\CC)$ of the direct summand $\wedge_{i \in S_j} e_i$ is the semidirect product
\begin{equation*}
U_S^j=\langle U_{\a_a-\a_b}:(a\neq b, a,b\in S_j) \text{ or }  (a\neq b, b\notin S_j)\rangle \rtimes T^{S_j}   
\end{equation*}
Now $j\in S_j$ for $1\le j \le n$ and by \eqref{transitivity} the intersection is 
\[\cap_{j=1}^n U_S^j=\langle U_{\a_a-\a_b}:b<a, b\in S_a\rangle \rtimes (\cap_{j=1}^n T^{S_j})\]
Since $j\in S_j$ and $S_j \subset \{1,\ldots, j\}$ for all $1\le j \le n$ we have by induction on $n$ that
\[\cap_{j=1}^n T^{S_j}=\{\mathrm{diag}(a_1,\ldots, a_n):a_1=1,\Pi_{i\in S_2}a_i=1,\ldots ,\Pi_{i\in S_n}a_i=1\}=1,\]
and therefore 
\[\cap_{j=1}^n U_S^j=\langle U_{\a_a-\a_b}:b<a, b\in S_a \rangle=U_S\]
by definition.
\end{proof}

\begin{corollary}
The map $\rho_S:\SL_n(\CC)\to \calw_S$ defined as $(v_1,\ldots, v_n) \mapsto \bigoplus_{j=1}^n \wedge_{i\in S_j} v_i$ on a matrix with column vectors $v_1,\ldots, v_n$ is invariant under the right multiplication action of $U_S$ on $\SL_n(\CC)$ and the induced map
\[\SL_n(\CC)/U_S \hookrightarrow \calw_S\]
on the set of $U_S$-orbits is injective and $\SL_n(\CC)$-equivariant with respect to the left multiplication action of $\SL_n(\CC)$ on $\SL_n(\CC)/U_S$.  
\end{corollary}

\begin{ex} This example shows that $\rho_S(\SL_n(\CC))=\SL_n(\CC) \cdot p_{S}$ might have codimension 1 boundary components in $\calw_S$. Take 
\[p_S=e_1 \oplus (e_1 \wedge e_2) \oplus (e_1 \wedge e_3) \oplus (e_1 \wedge e_2 \wedge e_3 \wedge e_4) \in \CC^4 \oplus \wedge^2 \CC^4 \oplus \wedge^2 \CC^4 \oplus \wedge^4 \CC^4\]
corresponding to the group 
\[U_S=\begin{array}{|cccc|}
\hline 
1 & \bullet & \bullet & \bullet \\
  & 1 & 0 & \bullet \\
  &   & 1 & \bullet \\
  &   &   & 1 \\
\hline
\end{array}.\]
Then 
\[\lim_{t\to 0} \left(\begin{array}{cccc}
t & 0 & 0 & 0\\
0  & t^{-1} & 0 & 0\\
 0 & 0  & t^{-1} & 0\\
0  & 0  & 0  & t
 \end{array}\right)\cdot p_S=0 \oplus (e_1 \wedge e_2) \oplus (e_1 \wedge e_3) \oplus (e_1 \wedge e_2 \wedge e_3 \wedge e_4),\]
whose stabiliser in $\SL_n(\CC)$ is
\[\left\{\left(\begin{array}{cccc}
a & b & c & d\\
0  & a^{-1} & 0 & e\\
0  & 0  & a^{-1} & f\\
0  & 0  & 0  & a
 \end{array}\right):a\in \CC^*, b,c,d,e,f\in \CC\right\}\]
which has dimension one plus the dimension of $U_S$, and therefore this boundary orbit has codimension $1$ in $\overline{\SL_n(\CC)\cdot p_{S}} \subset \calw_S$.

To make things even worse, in this example we have infinitely many boundary orbits, which means that it is not enough to study the boundary orbits and their stabilisers to prove the Grosshans property. Indeed, 
\[\lim_{t\to 0} \left(\begin{array}{cccc}
1 & 0 & 0 & 0\\
0  & 1 & \a & 0\\
 0 & 0  & t & 0\\
0  & 0  & 0  & t^{-1}
 \end{array}\right)\cdot p_S=e_1 \oplus (e_1 \wedge e_2) \oplus \a(e_1 \wedge e_2) \oplus (e_1 \wedge e_2 \wedge e_3 \wedge e_4),\]
and for different $\a$'s these boundary points sit in different $\SL_n(k)$ orbits.  
\end{ex}

Let $\tS=\{S_{i_1}\cup \ldots \cup S_{i_r}:1\le i_1<\ldots <i_r\le n\}\subset 2^{\{1,\ldots ,n\}}$ be the family of all possible unions of the $S_j$'s. Recall from \eqref{partialorder1} that the stars in the box form of $U_S$ form an incidence matrix of a partial order of $\{1,2,\ldots, n\}$, that is, 
\begin{equation}\label{partialorder}
i\in S_j,\ j\in S_k \Rightarrow i\in S_k 
\end{equation}
holds for $1\le i<j<k\le n$ and therefore for $i<j$
\[S_i \cap S_j=\cup_{k\in S_i\cap S_j} S_k.\] 
Hence $\tS$ is a so-called ring family, that is, closed under intersections and finite unions:   $U,V \in \tS$ implies $U\cap V,U \cup V \in \tS$. The point
\[p_{\tS}=\bigoplus_{U\in \tS} \wedge_{i\in U} e_i \in 
\calw_{\tS}\]
where 
\[\calw_{\tS}=\bigoplus_{U \in \tS} \wedge^{|U|}\CC^n.\]
has the same stabilizer as $p_S$, that is Theorem 
\ref{stabiliser} implies
\begin{corollary}
The stabilizer of $p_{\tS}$ in $\SL_n$ is $U_S$.
\end{corollary}

In our example \eqref{regularsubgroup} we have
\[\tS=\{\{1\},\{2\},\{1,3\},\{2,4\},\{1,2\},\{1,2,3\},\{1,2,4\},
\{1,2,3,4\}\}\]
and 
\[p_{\tS}=e_1 \oplus e_2 \oplus (e_1 \wedge e_3) \oplus (e_2 \wedge e_4) \oplus (e_1 \wedge e_2) \oplus (e_1 \wedge e_2 \wedge e_3) \oplus (e_1 \wedge e_2 \wedge e_4)\oplus (e_1 \wedge e_2 \wedge e_3 \wedge e_4).\]
\begin{conjecture}
$(\calw_{\tS},p_{\tS})$ is a Grosshans pair for the group $U_S \subset \SL_n(\CC)$, that is, the boundary components of the orbit $\SL_n(\CC) \cdot p_{\tS}$ have codimension at least $2$ in its closure: 
\[\dim(\overline{\SL_n(\CC) \cdot p_{\tS}}\setminus \SL_n(\CC) \cdot p_{\tS}) \le \dim(\overline{\SL_n(\CC) \cdot p_{\tS}})-2.\]
\end{conjecture}

\section{Left Borel-regular subgroups of $\SL_n(k)$}\label{blocksln}

In this section we prove Theorem \ref{main}. As we already mentioned in the introduction of \S\ref{sec:construction}, we only consider the $k=\CC$ case but the arguments work for any algebraically closed field $k$ of characteristic zero which is a subfield of $\CC$. 

Recall from the Introduction that a left Borel-regular subgroup $U_S\subset \SL_n(\CC)$ is determined by a sequence $\t=(\t_1,\ldots, \t_n)$ such that $0 \le \t_i<i$ and the corresponding closed root subset $S=(S_1,\ldots, S_n) \subseteq R^+$ has the form   
\[S_i=\begin{cases} \{1,\ldots \t_i,i\} & \text{ when } t_i >0 \\ \{i\} & \text{ when } t_i=0 \end{cases}
\]
For a subset $Z \subset \{1,\ldots, n\}$ we define $\max(Z)=\max_{z\in Z} z$ to be its maximal element and we let 
\[\t_Z:=\max_{z\in Z} \theta_z.\]

Elements of the generated ring family $\tilde{S}$ are formed by unions of the $S_i$'s and for a subset $Z \subset \{1,\ldots n\}$ the corresponding element of $\tS$ is
\begin{equation}\label{zform}
S_{Z}=\cup_{z\in Z} S_z=\{1,\ldots, \t_Z\} \cup \{z\in Z: z>\t_Z\}.
\end{equation}
That is, \eqref{zform} tells us that $S_Z$ contains all integers between $1$ and $\t_Z$ along with those elements of $Z$ which are bigger than $\theta_Z$. In other words, if $Z\in \tS$ with $\max(Z)=l$ then there exist integers  $\theta_Z<j_1<j_2<\ldots <j_s=l$ such that  
\begin{equation}\label{zform2}
Z=\{1,\ldots, \theta_Z,j_1, \ldots ,j_s=l\},
\end{equation}
see Figure \ref{figureone} for an example.

\begin{rem} Borel-regular subgroups are by definition left and right Borel-regular and they correspond to monotone increasing sequences $0\le \t_1 \le \ldots \le \t_n\le n-1$, therefore $\theta_Z=\theta_{\max(Z)}$ holds.
\end{rem}

\begin{figure}[h]
\begin{center}
\[U^{0,0,1,3,2}=
\begin{array}{|ccccc|}
S_1 & S_2 & S_3 & S_4 & S_5  \\
\hline 
1 & 0 & \bullet & \bullet & \bullet \\
  & 1  & 0 & \bullet  & \bullet   \\
  &   & 1 & \bullet  & 0  \\
  &   &   & 1 & 0  \\
  &   &   &  & 1  \\
\hline
\end{array}\]
\end{center}
\caption{The group corresponding to the sequence $\theta=(0,0,1,3,2)$. Elements of $\tS$ are $S_1=\{1\},S_2=\{2\},S_3=\{1,3\},S_4=\{1,2,3,4\},S_5=\{1,2,5\},S_1 \cup S_2=\{1,2\},S_2\cup S_3=\{1,2,3\}, S_3\cup S_5=\{1,2,3,5\}\text{ and } S_4 \cup S_5=\{1,2,3,4,5\}$.}  
\label{figureone}
\end{figure}

The boundary points in $\overline{B_n \cdot p_{\tilde{S}}}$ are limits of the form 
\begin{equation}\label{limitform}
p^\infty=\lim_{m\to \infty} 
\left( \begin{array}{cccc}
b_{11}^{(m)} & b_{12}^{(m)} & \cdots & b_{1n}^{(m)}\\
0 & b_{22}^{(m)} &   & \\
 & & \ddots & \\
 0 & & & b_{nn}^{(m)}
\end{array}\right)\cdot p_{\tS}.
\end{equation}
The sequence $(b^{(m)})$ in \eqref{limitform} is not unique: different sequences can define the same limit point $p^\infty$. However, any sequence $(b^{(m)})$ has a (not unique) subsequence $(b^{(m_s)})_{s=1}^\infty$ such that for all $1\le i,j \le n$ either $\lim_{s\to \infty} b^{(m_s)}_{ij}$ exists or the modulus $|(b^{(m_s)}_{ij})|$ tends to infinity. Then we can use this subsequence in \eqref{limitform} to define $p^\infty$, see definition \ref{def:vanish} below. 

Next observe that if $i \in S_j=\{1,\ldots, \theta_j,j\}$ for some $1\le i<j \le n$ then $b^{(m)}p_{\tS}$ is independent of the value of $b_{ij}^{(m)}$. Indeed, for a $Z\in \tS$ of the form \eqref{zform2} the expression for $b^{(m)}\cdot p_Z$ contains $b_{ij}^{(m)}$ only if $j\in Z$. However, if $j\in Z$ then $S_j=\{1,\ldots, \theta_j,j\}\subset Z$ and since $i \le \theta_j$, we have $i\in Z$.  Therefore 
\[b^{(m)}\cdot p_Z=\cdots \wedge \underbrace{(b^{(m)}_{jj}e_j+\ldots +b^{(m)}_{ij}e_i+\ldots)}_{b^{(m)}\cdot e_j}\wedge \cdots \wedge \underbrace{(b^{(m)}_{ii}e_i+b^{(m)}_{i-1i}e_{i-1}+\ldots )}_{b^{(m)}\cdot e_i} \wedge \cdots\] 
and $b^{(m)}_{ij}$ vanishes by taking the wedge product. 

This means that changing the value of $b_{ij}^{(m)}$ will not change the point $b^{(m)}p_{\tS}$ so we may assume without loss of generality that
\begin{equation}\label{condition}
b^{(m)}_{ij}=0 \text{ holds for all } m \text{ and } 1\le i< j \le n \text{ such that }  i\in S_j.
\end{equation}

\begin{definition}\label{def:vanish}
We call a sequence $(b^{(m)})$ of matrices \textbf{normalized} if it satisfies \eqref{condition} and for all $1\le i,j \le n$ either $b^\infty_{ij}=\lim_{m\to \infty} b^{(m)}_{ij} \in \CC$ exists or $|(b^{(m)}_{ij})| \to \infty$ as $m\to \infty$. In the latter case we write $b^\infty_{ij}=\infty$. 
The \textbf{vanishing spectrum} of the normalized sequence $(b^{(m)})$ is defined as
\[\mathrm{VSpec}(b^{(m)})=\{i: \lim_{m\to \infty}b_{ii}^{(m)}=0\},\]
\end{definition}

In short, the proof of Theorem \ref{main} will follow an induction argument on the size of the vanishing rank of the normalized sequence in \eqref{limitform} which defines $p^\infty$. 
\begin{definition}\label{def:bu}
Let $\mathcal{B}_{\bu}$ denote the set of boundary points in $\overline{B_n \cdot p_{\tilde{S}}}$ which are limits of the form \eqref{limitform} with vanishing diagonal entries indexed by the array $\bu=(u_1,\ldots, u_s)$, that is, 
\[\calb_\bu=\{p^\infty \in \overline{B_n \cdot p_{\tilde{S}}}:\exists \text{ normalized sequence } (b^{(m)}) \text{ s.t } p^\infty=\lim_{m\to \infty}b^{(m)} \cdot p_{\tilde{S}} \text{ and } \mathrm{VSpec}(b^{(m)})=\bu\}.\]
\end{definition}
As we noted above, every boundary point $p^\infty \in \overline{B_n \cdot p_{\tilde{S}}}$ is the limit of the form \eqref{limitform} for a normalized $(b^{(m)})$ and therefore 
\[\overline{B_n \cdot p_{\tilde{S}}}=\cup_{\bu \in 2^{\{1,\ldots ,n\}}}\mathcal{B}_{\bu}.\]

According to the next Lemma $\mathcal{B}_{\emptyset}=B\cdot p_{\tilde{S}}$ is the Borel orbit.

\begin{lemma}\label{borelorbit}
$\mathcal{B}_{\emptyset}=B\cdot p_{\tilde{S}}$ is equal to the Borel orbit of $p_{\tS}$ in $\calw_{\tS}$.
\end{lemma}

\begin{proof}[First proof] We use the following fact about solvable groups.

Let $H \subset \GL(V)$ be a solvable algebraic group and let $v\in V$. Then there is an $f\in k[V]$ and a character $\chi:H \to k^*$ such that $f(h\cdot w)=\chi(h)f(w)$ holds for all $b\in B, w\in V$ and $\overline{H \cdot v} \setminus H \cdot v=\{w\in \overline{H \cdot v}:f(w)=0\}$. 

We apply this result with the Borel $H=B_n \subset \SL(n)$. Let $(b^{(m)})$ be a sequence such that
\[p^\infty=\lim_{m\to \infty} b^{(m)}p_{\tS} \in \overline{B_n \cdot p_{\tS}}\setminus B_n \cdot p_{\tS}.\] 
Let $f$ be as above, vanishing on $\overline{B_n \cdot p_{\tS}} \setminus B_n \cdot p_{\tS}$. Then 
\[0=f(p^\infty)=\lim_{m\to \infty}f(b^{(m)}\cdot p_{\tS})=\lim_{m\to \infty}\chi(b^{(m)})f(p_{\tS})\]
Since $f(p_{\tS})\neq 0$, we must have $\lim_{m\to \infty}\chi(b^{(m)})=0$ so $\lim_{m\to \infty}\chi_{ii}(b^{(m)})=0$ for some $1\le i \le n$. 
\end{proof}

We give a second, longer proof because its main idea will turn up in the proof of Theorem \ref{mainc} in \S\ref{strategy} again. 

\begin{proof}[Second proof]
Let $p^\infty=\lim_{m\to \infty} b^{(m)}p_{\tS} \in \calb_{\emptyset}$ be defined by the normalized sequence $(b^{(m)})$ such that  
$\mathrm{VSpec}(b^{(m)})=\emptyset$. By definition for all $1\le i \le n$ either $b^\infty_{ii}=\lim_{m\to \infty} b^{(m)}_{ii}\in \CC$ exists or $\lim_{m\to \infty}|b^{(m)}_{ii}|=\infty$. 
Since
\[b_{11}^{(m)}\cdot \ldots \cdot b_{nn}^{(m)}=1 \text{ holds for all } m,\]
if $b^{(m)}_{ii} \to \infty$ for some $1\le i \le n$ then for some $j\neq i$ $\lim_{m\to \infty} b^{(m)}_{jj}=0$ so $j\in\mathrm{VSpec}(b^{(m)})$, a contradiction. This proves that   
\[b^\infty_{ii}=\lim_{m\to \infty} b^{(m)}_{ii}\in \CC \setminus \{0\} \text { for } 1\le i \le n.\]

Assume that some off-diagonal entries of the normalized sequence $(b^{(m)})$ are not convergent and let
\[v=\min\{j: \exists i\notin S_j \text{ such that } i<j  \text{ and } b^{\infty}_{ij}=\infty\}\]
be the leftmost column containing such entries and choose a $u<v, u\notin S_v$ such that $b^{\infty}_{uv}=\infty$. Then due to \eqref{condition} the coefficient of $e_u \wedge (\wedge_{i\in S_v\setminus \{v\}}e_i)$ in $p_{S_v}^\infty=\lim_{m\to \infty} \wedge_{i\in S_v}b^{(m)}e_i$ would be
\[\lim_{m\to \infty}b_{uv}^{(m)}\cdot \prod_{i\in S_v\setminus \{v\}}b^\infty_{ii}=\infty,\]
a contradiction.  
Hence $b^\infty_{ij}:=\lim_{m\to \infty}b^{(m)}_{ij}\in \CC$ exists for all $1\le i\le j \le n$ and $b_{ii}^{\infty}\neq 0$ so $b^\infty=\lim_{m\to \infty}b^{(m)}\in B_n$ exists and
\[p^\infty=\lim_{m\to \infty}b^{(m)}\cdot p_{\tS}=b^\infty \cdot p_{\tS} \in B_n\cdot p_{\tS}.\]
\end{proof}



\begin{definition}\label{def:covered}
Let $U_S=U^{\t} \subset \SL_n$ be a left Borel-regular subgroup and $Z\subset \{1,\ldots, n\}$. We say that the integer $1\le u < n$ is \textbf{covered by} $Z$ if $u\le \t_Z$. We say that $u$ is covered by $S$ if it is covered by  at least one of $S_1,\ldots, S_n$, that is, $u\le \max_{1\le i \le n} \t_i$. The subset $\bu=\{u_1<\ldots <u_s\} \subset \{1,\ldots, n\}$ is covered by $S$ if all elements of it are covered by $S$. 
 \end{definition}
 
\begin{exit} In Figure \ref{figureone} $S_4$ covers $1,2,3$ and $S_5$ covers $1$ and $2$. Moreover, $u=1,2,3$ are covered by $S$, $u=4$ is not covered. Therefore any subset $\bu$ of $\{1,2,3\}$ is covered by $S$ and the subsets containing $4$ are not covered by $S$.  
\end{exit}

Let $Z \in \tilde{S}$ and let $b^{(m)}$ be a normalized sequence. In what follows we will work with subspaces of $\CC^n$ determined by $b^{(m)}\cdot p_Z$ and the limit of these subspaces.  
\begin{definition}  For the nonzero vectors $v_1,\ldots, v_s \in \CC^n$ let 
\[[v_1\wedge \ldots \wedge v_s]\in \grass_s(\CC^n)\] 
denote the subspace spanned by them. In particular, for a subset $Z\in \tS$ we let
\[[p_Z]=[\wedge_{z\in Z} e_z] \in \grass_{|Z|}\CC^n\]
and 
\begin{equation}\label{defpzinfty}
[p_Z^\infty]=\lim_{m\to \infty} [b^{(m)}\cdot p_Z]
\end{equation}
denotes the limit in $\grass_{|Z|}(\CC^n)$.
\end{definition}

\begin{rem}
If $p^\infty_Z=\lim_{m\to \infty} b^{(m)}\cdot p_Z \in \wedge^{|Z|}\CC^n$ exists then either $p^\infty_Z=0$ or it is a decomposable vector, i.e. $p^\infty_Z=w_1 \wedge \ldots \wedge w_{|Z|}$ for some nonzero vectors $w_1,\ldots, w_{|Z|}$. To see this, note that the Veronese map 
\[\mu: \grass_{|Z|}(\CC^n) \hookrightarrow \PP(\wedge^{|Z|}\CC^n)\]
is a closed embedding and the set of decomposable vectors in $\wedge^{|Z|}\CC^n$ forms the affine cone over the image of $\mu$ and therefore this set is closed. Hence the limit of decomposable elements in $\wedge^n \CC^n$ is either decomposable or zero. 
Moreover, if $0\neq p^\infty_Z=w_1 \wedge \ldots \wedge w_{|Z|}$ then the subspace $[w_1 \wedge \ldots \wedge w_{|Z|}]$ is equal to the limit defined in \eqref{defpzinfty}. 
\end{rem}

The following technical lemma will be used repeatedly in this section. 

\begin{lemma}\label{firstmu}
Let $\bu=\{u_1<\ldots <u_s\}\subset \{1,\ldots, n\}$ be an arbitrary subset. Let 
\[p^\infty =\lim_{m\to \infty} b^{(m)} p_{\tS}=\oplus_{U\in \tilde{S}} p^\infty_{U} \in \calb_\bu\]
 be a limit point defined by the normalized sequence $(b^{(m)})$ such that $\mathrm{VSpec}(b^{(m)})=\bu$. If $\t_i<u_1$ for some $1\le i \le n$ then $\lim_{m\to \infty}b^{(m)}e_i\in \CC^n$ exists.
\end{lemma}

\begin{proof}
Assume $\t_i < u_1$ but $\lim_{m\to \infty}|b^{(m)}e_i| \to \infty$ and $i$ is the smallest index with this property so $e_j^\infty=\lim_{m\to \infty}b^{(m)}e_j \in \CC^n$ for $1\le j<i$. Then 
\[p_{S_i}^\infty=\wedge_{j=1}^{\t_i}e^\infty_j \wedge \lim_{m\to \infty} b^{(m)}e_i \]
does not exist, a contradiction. 
\end{proof}

\begin{rem}\label{u1determined}
Note that $p^\infty=\bigoplus_{V \in \tS}p_V^\infty=\lim_{m\to \infty}b^{(m)}p_{\tS}$ can be an element of several different $\calb_\bu$'s, that is, $\bu=\{u_1<\ldots <u_s\}$ is not uniquely determined by $p^\infty$. However, Lemma \ref{firstmu} implies that $p^\infty$ determines $u_1$ at least:  $u_1$ signs  the first diagonal entry of $b^{(m)}$ which tends to $0$ as $m \to \infty$ and all previous diagonal entries tend to a nonzero constant. Therefore  
\[u_1=\min\{j: p_{\{1,\ldots, j\}}^\infty \subset \mathrm{Span}(e_1,\ldots, e_{j-1}).\]
\end{rem}

\begin{lemma}\label{mustbecovered} Let $\bu \subset \{1,\ldots, n\}$ a subset and assume $\calb_{\bu}$ is non-empty. Then $\bu$ is covered by $S$.  
\end{lemma}

\begin{proof}
Let $p^\infty=\lim_{m\to \infty} b^{(m)}p_{\tS} \in \calb_{u}$ be a point defined by a normalized sequence $(b^{(m)})$ with $\mathrm{VSpec}(b^{(m)})=\{u_1<\ldots <u_s\}$. Assume that $u_s>\max_{1\le i \le n} \t_i$. Then
\[V=\{1,\ldots, n\}\setminus \{u_s\}=\cup_{i \in \{1,\ldots, n\}\setminus \{u_s\}}S_i \in \tS\]
and the coefficient of $\bigwedge_{i\in V} e_i$ in $p^\infty_V$ is 
\[p^\infty_V[\wedge_{i\in V} e_i]=\lim_{m\to \infty} \prod_{i\in V}b^{(m)}_{ii}=\lim_{m\to \infty} \frac{1}{b_{u_su_s}^{(m)}}=\infty,\]
a contradiction. Here we used that $b^{(m)}\in \SL_n$ and hence $\prod_{i=1}^n b^{(m)}_{ii}=1$ for all $m$.

\end{proof}

\begin{definition}\label{def:important}  Let $\bu \subset \{1,\ldots, n\}$ be a subset covered by $S$ and let $p^\infty =\lim_{m\to \infty} b^{(m)} p_{\tS}=\oplus_{U\in \tilde{S}} p^\infty_{U} \in \calb_\bu$ be a boundary point defined by the normalized sequence $(b^{(m)})$ with $\mathrm{VSpec}(b^{(m)})=\bu$. Let $\calv(p^\infty)$ denote the set of those elements in $\tS$ which cover $u_1$ and the corresponding term of $p^\infty$ is nonzero, that is 
\[\calv(p^\infty)=\{U \in \tS: p^{\infty}_U\neq 0, \t_U \ge u_1\}.\]
According to Remark \ref{u1determined} this set is determined by $p^\infty$. It is non-empty: $\{1,\ldots, n\}\in \calv(p^\infty)$ because $\{1,\ldots, n\}$ covers $\bu$ and
\[p_{\{1,\ldots ,n\}}^{\infty}=\lim_{m\to \infty} \det(b^{(m)})e_1 \wedge \ldots \wedge e_n=e_1 \wedge \ldots \wedge e_n \neq 0.\] 
Now define the partial order $\preceq$ on the elements of $\calv(p^\infty)$ 
as follows. For $U,V \in \calv(p^\infty)$ we write $U\preceq V$ if $\t_U<\t_V$ or $\t_U=\t_V$ but $U\subseteq V$. 
Let $\calv(p^\infty)_{\min}$ denote the set of minimal elements of $\calv(p^\infty)$ with respect to $\preceq$. We call elements of $\calv(p^\infty)_{\min}$ minimal for $p^\infty$. 
\end{definition}

A central part of our argument is the following technical proposition.
\begin{prop}\label{firstprop}
Let $\bu \subset \{1,\ldots, n\}$ be a subset covered by $S$ and $p^\infty \in \calb_\bu$. If $Z\in \calv(p^\infty)_{\min}$ then  
\begin{equation}\label{firstpropeq}
[p^{\infty}_Z] \subset \bigcap_{V \in \calv(p^\infty)}[p^{\infty}_{V}]
\end{equation}
\end{prop}

\begin{proof} By definition we can write $p^\infty =\lim_{m\to \infty} b^{(m)} p_{\tS}=\oplus_{U\in \tilde{S}} p^\infty_{U} \in \calb_\bu$ as a limit point where $(b^{(m)})$ is normalized and $\mathrm{VSpec}(b^{(m)})=\bu$.
By definition $Z \in \calv(p^\infty)_{\min}$ satisfies the following properties: 
\begin{enumerate}
\item $Z\in \calv(p^\infty)$, that is, $p^{\infty}_Z\neq 0$ and $\t_Z \ge u_1$. 
\item If $p^{\infty}_U \neq 0$ and $\t_U \ge u_1$ for some $U \in \tS$ then $\t_U \ge \t_Z$ holds.  
\item If $U\in \tS$, $\t_U \ge u_1$ and $U \subsetneqq Z$ then $p^{\infty}_{U}=0$. 
\end{enumerate}

Here (ii) and (iii) together say that $Z\in \calv(p^\infty)_{\min}$. Assume there is a $V \in \calv(p^\infty)$ such that $[p^{\infty}_{Z}] \nsubseteq [p^{\infty}_{V}]$. By definition 
\[p^\infty_Z =\lim_{m\to \infty} \wedge_{z\in Z}b^{(m)} e_{z} \text{ and } p^\infty_V =\lim_{m\to \infty} \wedge_{v\in V}b^{(m)} e_{v},\]
therefore $Z \subset V$ would imply that $[p^{\infty}_{Z}] \subseteq [p^{\infty}_{V}]$. So $Z \setminus V$ must be non-empty. 

Fix a hermitian form $(\cdot, \cdot)$ on $\CC^n$ and let $\pi_V:\CC^n \to [p^\infty_V]$ denote the projection to the subspace $[p^\infty_V]$. For $w\in \CC^n$ let  $w^\perp=w-\pi_V(w)$ denote the orthogonal component. 

In Corollary \ref{corollarytechnical} below  we show that if we drop any subset $\emptyset \neq \Gamma \subseteq Z\setminus V$ from $Z$ then the smaller subset $Z \setminus \Gamma$ is still in $\tS$.  We claim that 
\begin{equation}\label{eqgamma}
p^\infty_{Z\setminus \Gamma}=0 \text{ for all } \emptyset \neq \Gamma \subseteq Z\setminus V.
\end{equation}
Indeed, if $\t_{Z\setminus \Gamma} \ge u_1$ then this is property (iii) above. If $\t_{Z\setminus \Gamma}<u_1$ then, by Lemma \ref{firstmu}, 
\begin{equation}\label{ezinfty}
e_{z}^\infty=\lim_{m\to \infty}b^{(m)}e_{z}\in \CC^n \text{ exists for all } z \in Z\setminus \Gamma.
\end{equation}

 Moreover, $\t_Z \ge u_1$ hence $\{1,2,\ldots, u_1\} \subset \{1,2,\ldots, \theta_Z\} \subset Z$. On the other hand property (ii) tells us that $\t_V \ge \t_Z$ and therefore 
 \[\{1,2,\ldots, \theta_Z\} \subseteq \{1,2,\ldots, \theta_V\} \subset V\]
 which implies that 
 \[\{1,2,\ldots, u_1\} \subset Z \setminus \Gamma\]
and therefore 
\[p^\infty_{Z\setminus \Gamma}=e_{u_1}^\infty \wedge (\wedge_{z \in Z\setminus \{\Gamma \cup \{u_1\}\}}e_z^\infty)=b^{\infty}_{u_1u_1}e_{u_1} \wedge b^{\infty}_{u_1-1u_1-1}e_{u_1-1}\wedge \ldots \wedge  b^{\infty}_{11}e_{1} \wedge(\wedge_{z \in Z\setminus \{\Gamma \cup \{1,\ldots, u_1\}\}}e_z^\infty.\]
But $b^{\infty}_{u_1u_1}=0$ and by \eqref{ezinfty} all other terms are finite so this wedge product is $0$ and \eqref{eqgamma} is proved.  

Then \eqref{eqgamma} implies that
\begin{equation}\label{piv}
0=\pi_V(p^\infty_{Z\setminus \Gamma})=\lim_{m\to \infty} \wedge_{j\in Z\setminus \Gamma} \pi_V(b^{(m)}\cdot e_j) \text{ for all } \emptyset \neq \Gamma \subseteq Z\setminus V
\end{equation}
Hence 
\begin{multline}\nonumber 
p^\infty_{Z}=\lim_{m\to \infty} \bigoplus_{\Gamma \subseteq Z\setminus V} \wedge_{j \in \Gamma} b^{(m)}e_j^{\perp}\bigwedge \wedge_{j \in Z \setminus \Gamma} \pi_V(b^{(m)}\cdot e_j)=\\
=\pi_V(p^\infty_Z)\oplus \lim_{m\to \infty} \bigoplus_{\emptyset \neq \Gamma \subseteq Z\setminus V} \wedge_{j \in \Gamma} b^{(m)}e_j^{\perp}\bigwedge \wedge_{j \in Z \setminus \Gamma} \pi_V(b^{(m)}\cdot e_j)
\end{multline}
By \eqref{piv} all terms corresponding to nonempty $\Gamma$ vanish and therefore $[p^{\infty}_{Z}] \subseteq [p^{\infty}_{V}]$ unless there is a $z \in Z\setminus V$ such that the limit norm $\lim_{m\to \infty}|b^{(m)}e_z^\perp|=\infty$. 

However, $V \cup \{z\}=V \cup S_z$ because $z\in Z$ and $S_z=\{1,\ldots, \theta_z,z\} \subseteq \{1,\ldots, \theta_Z,z\}\subseteq \{1,\ldots, \theta_V,z\}$ since $\theta_Z \le \theta_V$ by property (ii). Therefore $V \cup \{z\}\in \tS$ and then 
\[p^\infty_{V \cup \{z\}}=\lim_{m\to \infty} b^{(m)}e_z^\perp \wedge p^\infty_{V}\]
does not exist (the limit is not finite) which is a contradiction. So $[p^{\infty}_{Z}] \subseteq [p^{\infty}_{V}]$ holds, and Proposition \ref{firstprop} is proved. 
\end{proof}

\begin{corollary}\label{pzinftyisdetermined}
Assume $Z_1,Z_2 \in \calv(p^\infty)_{\min}$ are minimal subsets for $p^\infty$. Then $[p_{Z_1}^\infty]=[p_{Z_2}^\infty]$.
\end{corollary}

\begin{proof}
As $Z_1$ and $Z_2$ are both in $\calv(p^\infty)$, by Proposition \ref{firstprop} we have $[p_{Z_1}^\infty] \subseteq  [p_{Z_2}^\infty]$ and $[p_{Z_2}^\infty] \subseteq [p_{Z_1}^\infty]$.
\end{proof}

Here are the small technical statements on $\tS$ we used in the proof of Proposition \ref{firstprop}.

\begin{lemma}\label{dropelements}  Let $Z\in \tS$ and let $z \in Z$ be an element such that $z>\theta_Z$. Then $Z\setminus \{z\} \in \tS$. 
\end{lemma}
\begin{proof}
Enough to show that $Z \setminus \{z\}=\cup_{i\in Z \setminus \{z\}} S_i$. The direction $\subseteq$ is clear as $i\in S_i$ for all $i$. For $\supseteq$ note that if $i\in Z \setminus \{z\}$ and $z>\t_Z$ then
\[S_i=\{1,2,\ldots, \theta_i\}\cup \{i\} \subset \{1,2,\ldots, \theta_Z\}\cup \{i\} \subseteq Z \setminus \{z\}\]
\end{proof}

\begin{corollary}\label{corollarytechnical}
 If $Z,V \in \tS$ such that $\t_V \ge \t_Z$ then for any $\Gamma \subseteq Z\setminus V$ we have $Z \setminus \Gamma \in \tS$.
\end{corollary}
\begin{proof}
Since $\theta_V \ge \theta_Z$, any $z\in Z\setminus V$ must satisfy $z>\theta_Z$ and the statement follows from Lemma \ref{dropelements}. 
\end{proof}

\begin{definition}\label{def:bur} Let $Z\in \calv(p^\infty)_{\min}$ be a minimal subset for $p^\infty$. By Corollary \ref{pzinftyisdetermined} the subspace $[p_Z^\infty]$ is independent of the choice of $Z$ and depends only on $p^\infty$. Assume that $[p^\infty_{Z}] \subset \mathrm{Span}(e_1,\ldots, e_r)$ but $[p^\infty_{Z}] \not\subset \mathrm{Span}(e_1,\ldots, e_{r-1})$ for some $r$. We call this $r$ the \textbf{width} of $p^\infty$ and denote it by $\omega(p^\infty)$. We will also say that $[p^\infty_{Z}]$ has width $r$. Since $u_1\le\t_Z$ and $\{1,\ldots, \t_Z\} \subsetneq Z$ this must  then satisfy $u_1\le \t_Z <\omega(p^\infty)$. 
Let  
\[\calb_{\bu}^r=\{p^\infty \in \calb_{\bu}: \omega(p^\infty)=r\}\]
denote the set of points in $\calb_{\bu}$ of width $r$.
\end{definition}

Then we have a (not necessarily disjoint) finite decomposition
\[\calb_{\bu}=\cup_{u_1 < r}\calb_{\bu}^r.\] 

\begin{rem}\label{remark:basechange}
Points of the boundary sets $\calb_\bu^r$ are defined as limits of normalized sequences and hence the stratification $\calb_{\bu}=\cup_{u_1 < r}\calb_{\bu}^r$ a priori depends on the choice of the basis $\{e_1,\ldots, e_n\}$ of $\CC^n$. Let us indicate this dependence temporarily as $\calb_\bu^r(e_1,\ldots, e_n)$. We show that changing this basis with a unipotent element of the Borel $B_{n}$ leaves all $\calb_\bu^r$ unchanged. 
More precisely, let $A\in  B_n$ define the new basis
\[\bar{e}_i=A \cdot e_i \text{ for } i=1\ldots, n\] 
Let $p_{\tS}(\bar{e}_1,\ldots, \bar{e}_n)$ denote the base point $p_{\tS}$ written in the new basis. Then  
\[p^\infty=\lim_{m\to \infty} b^{(m)} p_{\tS}=\lim_{m\to \infty} (A b^{(m)} A^{-1}) p_{\tS}(\bar{e}_1,\ldots, \bar{e}_n)\] 
If $A$ is unipotent, then the new sequence $\bar{b}^{(m)}=A b^{(m)} A^{-1}$ 
has the same vanishing spectrum $\bu$. Moreover, since
\[\mathrm{Span}(e_1,\ldots, e_i)=\mathrm{Span}(\bar{e}_1,\ldots, \bar{e}_i) \text{ for } 1\le i \le n,\]
the width of $[p_Z^\infty]$ in this neq basis is $r$ again. Therefore 
\[\calb_\bu^r(e_1,\ldots e_n)=\calb_\bu^r(\bar{e}_1,\ldots, \bar{e}_n) \text{ for all } \bu, r.\]
In short, changing the basis with a unipotent element of the Borel will leave the subsets $\calb_\bu^r$ unchanged. 
\end{rem}

\begin{rem}\label{conditiononer} Let $p^\infty \in \calb^r_\bu$ and $Z\in \calv(p^\infty)_{\min}$. By definition this means that $[p^\infty_{Z}] \subset \mathrm{Span}(e_1,\ldots, e_r)$ but $[p^\infty_{Z}] \not\subset \mathrm{Span}(e_1,\ldots, e_{r-1})$ so there is a vector 
 \[w=e_r+w_{r-1}e_{r-1}+\ldots +w_1e_1 \in [p^\infty_{Z}]\]
This $w$ is not necessarily unique, we fix one. The base change 
\[\overline{e}_i=\begin{cases} e_i & \text{ if } i\neq r \\ e_r+w_{r-1}e_{r-1}+\ldots +w_1e_1 & \text{ if } i=r\end{cases}\] 
is defined by a unipotent element of $B_{n}$. According to Remark \ref{remark:basechange} changing $\{e_1,\ldots, e_n\}$ to the new basis $\{\overline{e}_1,\ldots, \overline{e}_n\}$ leaves the boundary sets $\calb_\bu^r$ unchanged for all $\bu$ and $r$ but in this new basis $\bar{e}_r=w \in [p_Z^\infty]$ holds. 
\end{rem}

\begin{prop}\label{crucial}
Let $s\ge 2$ and $\bu=\{u_1<\ldots <u_s\}$ be a subset covered by $S$. Then  
\begin{enumerate}[(a)]
\item If $r \notin \bu$ and $\t_r <u_1$ then there is a continuous injection $\rho:\calb_{\bu}^r \hookrightarrow \calb_{\{u_2,\ldots, u_s\}\cup \{r\}}$ and therefore $\dim \calb_{\bu}^r \le \dim \calb_{\{u_2,\ldots, u_s\}\cup \{r\}}$.
\item If $r\in \bu$ then $\calb_{\bu}^r \subset \overline{\calb_{\{u_2,\ldots, u_s\}}}\setminus \calb_{\{u_2,\ldots, u_s\}}$.  
\item If $r\notin \bu$ and $\t_r \ge u_1$ then $\calb_{\bu}^r \subset \overline{\calb_{\{u_2,\ldots, u_s\}}}\setminus \calb_{\{u_2,\ldots, u_s\}}$ or $\calb_{\bu}^r \subset \overline{\calb_{\{u_2,\ldots, u_s\}\cup \{r\}}}\setminus \calb_{\{u_2,\ldots, u_s\}\cup \{r\}}$.
\end{enumerate}
\end{prop}

\begin{proof} 
To prove (a) assume that $r \notin \bu$ and $\t_r <u_1$. Let 
\[p^\infty =\lim_{m\to \infty} b^{(m)} p_{\tS}=\oplus_{U\in \tilde{S}} p^\infty_{U} \in \calb^r_\bu\]
be a limit point such that $\mathrm{VSpec}(b^{(m)})=\bu$ and $Z \in \calv(p^\infty)_{\min}$. By Remark \ref{conditiononer} we can assume that $e_r \in [p^\infty_{Z}]$.
 
 According to Lemma \ref{firstmu} $\lim_{m\to \infty}b^{(m)}_{ii}\in \CC$ exists whenever $\t_i <u_1$ and since $r \notin \bu$, this limit is nonzero for $i=r$: 
\[b_{rr}^\infty:=\lim_{m\to \infty}b^{(m)}_{rr} \in \CC \setminus \{0\}.\] 
Define the modified sequence
\begin{equation}\label{modifiedsl}
\tilde{b}_{ij}^{(m)}=\begin{cases} b_{rr}^\infty  & (i,j)=(u_1,u_1) \\ \frac{1}{b_{rr}^\infty}\cdot b^{(m)}_{rr}\cdot b^{(m)}_{u_1u_1} & (i,j)=(r,r) \\ \frac{1}{b_{rr}^\infty} \cdot b_{rj}^{(m)}\cdot b_{u_1,u_1}^{(m)} & \text{ for } i=r \text{ and } j> r \text{ with } \t_j \ge u_1\\ b_{ij}^{(m)} & \text{otherwise} \end{cases}
\end{equation}
In short, we fix the diagonal entry $b_{u_1u_1}^{(m)}$ to be the nonzero constant $b_{rr}^\infty$ and multiply the entries in the $r$th row of those columns which cover $b_1$ by $\frac{1}{b_{rr}^\infty} b_{u_1u_1}^{(m)}$.
Then the new sequence still sits in $\SL_n(\CC)$ and part (a) of Proposition \ref{crucial} follows from the following three statements:
\begin{enumerate}
\item $\tilde{p}^\infty=\lim_{m\to \infty} \tilde{b}^{(m)}p_{\tS}$ exists and therefore the map $\tilde{\rho}: p^\infty \mapsto \tilde{p}^\infty$ is well-defined. 
\item $\tilde{p}^{\infty} \in \calb_{\{u_2,\ldots, u_s\}\cup \{r\}}$
\item $\tilde{\rho}: p^\infty \mapsto \tilde{p}^\infty$ is injective. 
\end{enumerate}

To prove (i) and (ii) we first show that  
\begin{equation}\label{firsttoprove}
\text{ if } \t_V \ge u_1 \text{ then }  \tilde{p}^\infty_V=p^\infty_V.
\end{equation}
Note that in this case $u_1 \in \{1,2,\ldots, \t_V\} \subseteq V$ and therefore 
\[p^{\infty}_V=\lim_{m\to \infty}\wedge_{v\in V}b^{(m)}e_v=\lim_{m\to \infty} \left(\Pi_{i=1}^{\t_V}b^{(m)}_{ii} \bigwedge_{i=1}^{\t_V}e_i \wedge \bigwedge_{\t_V<v\in V}b^{(m)}e_v\right).\]
First we study the case when $r \le \t_V$. If $v>\t_V \ge r$ then $b^{(m)}e_v=\tilde{b}^{(m)}e_v$ and therefore the second product remains the same by changing $b^{(m)}$ to $\tilde{b}^{(m)}$. The product $\Pi_{i=1}^{\t_V}b^{(m)}_{ii}$ of the first $\t_V$ diagonal entries in $b^{(m)}$ and $\tilde{b}^{(m)}$ are again equal, so the first product does not change either giving us $\tilde{p}^\infty_V=p^\infty_V$.

Assume now that $r > \t_V$ and recall that $e_r\in [p_Z^{\infty}]\subset [p^\infty_V]$. Hence the term $e_r$ must be selected in each term of the expansion of the second product, that is, if $\pi^r:\CC^n \to \CC e_r$ denotes the projection of a vector to the line spanned by $e_r$ then $e_r \subset [p^\infty_V]$ implies that 
\begin{multline}\label{expansion}
p^{\infty}_V=\lim_{m\to \infty} \sum_{v \in V} \left(\pi^r(b^{(m)} e_v) \wedge \bigwedge_{i \in V\setminus v} b^{(m)}e_i\right)= \\\lim_{m\to \infty} \Pi_{i=1}^{\t_V}b^{(m)}_{ii} \bigwedge_{i=1}^{\t_V}e_i \wedge \sum_{r\le v\in V} \left(b^{(m)}_{rv} e_r \wedge \bigwedge_{\t_V<i \in V\setminus \{v\}}b^{(m)}e_i\right) .
\end{multline}
It is easy to see that if $r\le v$ and $\t_v < u_1$ then the corresponding term in the direct sum on the right hand side has zero contribution  Indeed, by Lemma \ref{firstmu} 
\begin{equation}\label{offdiagonalzero}
b^\infty_{rv}=\lim_{m\to \infty}b^{(m)}_{rv}\in \CC \text{ for } r\le v, \t_v <u_1.
\end{equation}
On the other hand $v\ge r>\t_V$ holds and therefore by Lemma \ref{dropelements} $V \setminus \{v\}\in \tS$. Moreover, since $\t_V >u_1$ but $\t_v <u_1$, we must have $\t_{V\setminus \{v\}} \ge u_1$. Now 
\begin{itemize}
\item if $p^\infty_{V\setminus \{v\}}=0$ then by \eqref{offdiagonalzero} we have $b^{\infty}_{rv}e_r \wedge p^\infty_{V\setminus \{v\}}=0$,
\item if $p^\infty_{V\setminus \{v\}}\neq 0$ then by Proposition \ref{firstprop} $e_r \in [p_Z^\infty] \subset [p^\infty_{V\setminus \{v\}}]$ and therefore by \eqref{offdiagonalzero} we have $b^{\infty}_{rv}e_r \wedge p^\infty_{V\setminus \{v\}}=0$ again. 
\end{itemize} 
In both cases we get 
\[0=b^{\infty}_{rv}e_r \wedge p^\infty_{V\setminus \{v\}}=\lim_{m\to \infty} (\Pi_{i=1}^{\t_V}b^{(m)}_{ii}) \bigwedge_{i=1}^{\t_V}e_i \wedge \left(b^{\infty}_{rv}e_r \wedge \bigwedge_{\t_V<i \in V\setminus \{v\}}b^{(m)}e_i\right).\]
So in \eqref{expansion} only those terms have nonzero contributions where $\t_v \ge u_1$ and 
\[p^{\infty}_V=\lim_{m\to \infty} \left(\Pi_{i=1}^{\t_V}b^{(m)}_{ii} \bigwedge_{i=1}^{\t_V}e_i \wedge \sum_{\substack{\t_v\ge u_1\\ r\le v \in V}} \left(b^{(m)}_{rv} e_r \wedge \bigwedge_{\t_V<i \in V\setminus \{v\}}b^{(m)}e_i\right)\right) .
\]
Replacing $b^{(m)}$ with $\tilde{b}^{(m)}$ clearly does not change the right hand side, so $\tilde{p}^\infty_V=p^\infty_V$ is proved for this case too. We completed the proof of \eqref{firsttoprove}.

Next, if $V \in \tS$ but $\t_V < u_1$ then $e_v^\infty=\lim_{m\to \infty}b^{(m)}e_v\in \CC^n$ exists for all $v\in V$ by Lemma \ref{firstmu} and therefore
\[p^\infty_V=\lim_{m\to \infty}\wedge_{v\in V}b^{(m)}e_v=\wedge_{v\in V}e_v^\infty.\]
Similarly 
\[\tilde{p}^\infty_V=\lim_{m\to \infty}\wedge_{v\in V}\tilde{b}^{(m)}e_v=\wedge_{v\in V} \tilde{e}_v^\infty.\]
where 
\begin{equation}\label{tildebmev}
\tilde{e}_v^\infty=\lim_{m\to \infty}\tilde{b}^{(m)}e_v=\begin{cases} e_v^\infty & v\neq u_1,r \\ e_v^\infty+b_{rr}^\infty e_{u_1} & v=u_1 \\ \pi^{r-1}(e_{r}^\infty) & v=r\end{cases}
\end{equation}
where $\pi^{r-1}:\CC^n \to \mathrm{Span}(e_1,\ldots, e_{r-1})$ is the projection. This is because $\lim_{m\to \infty} b^{(m)}_{u_1,u_1}=0$ by definition and the last coordinate of $b^{(m)}e_r$ tends to $0$:
\[\lim_{m \to \infty} \frac{1}{b_{rr}^\infty}\cdot b^{(m)}_{rr}\cdot b^{(m)}_{u_1u_1}=0.\] 
This means that $\tilde{p}^{\infty}_V=\lim_{m\to \infty} \tilde{b}^{(m)}p_V$ exists for all $V\in \tS$ and $\tilde{p}^{\infty} \in \calb_{\{u_2,\ldots, u_s\}\cup \{r\}}$, so (i) and (ii) is proved.  

To prove (iii) (the injectivity of $\tilde{\rho}:p^\infty \mapsto \tilde{p}^\infty$) note that by \eqref{firsttoprove} 
\[p_V^\infty = \tilde{p}^\infty_V \text{ whenever } \t_V\ge u_1\] 
so $\tilde{\rho}$ is the identity (and therefore injective) on these coordinates. It remains to check injectivity on the other coordinates.

Take two points $p^\infty \neq (p')^\infty$ in  $\calb_\bu^r$ such that $p_V^\infty \neq (p'_V)^\infty$ for some $V$ with $\t_V < u_1$. This means that  $e_v^\infty \neq (e'_v)^\infty$ for some $v \in V$ satisfying $\t_v<u_1$. Let $v$ be minimal with this property. 
Then 
\[p_{S_v}^\infty-(p'_{S_v})^\infty=(e_v^\infty-(e'_v)^\infty)\wedge \bigwedge_{i=1}^{\t_v}b_{ii}^\infty e_i\neq 0\]
and using \eqref{tildebmev} we have the following cases:
\begin{itemize}  
\item If $v\neq u_1,r$ then $\tilde{p}_{S_v}^\infty-(\tilde{p}'_{S_v})^\infty=p_{S_v}^\infty-(p'_{S_v})^\infty \neq 0$.
\item If $v=u_1$ then $\tilde{p}_{S_{u_1}}^\infty-(\tilde{p}'_{S_{u_1}})^\infty=p_{S_{u_1}}^\infty-(p'_{S_{u_1}})^\infty+(b_{rr}^\infty-(b'_{rr})^\infty)e_{u_1}\wedge \bigwedge_{i=1}^{\t_{u_1}}b_{ii}^\infty e_i \neq 0$ because $p_{S_{u_1}}^\infty-(p'_{S_{u_1}})^\infty$ does not contain $e_{u_1}$ due to the fact that $b^\infty_{u_1,u_1}=0$.
\item If $v=r$ then $\tilde{p}_{S_r}^\infty-(\tilde{p}'_{S_r})^\infty=\pi^{r-1}(e_r^\infty-(e'_r)^\infty)\wedge \bigwedge_{i=1}^{\t_r}b_{ii}^\infty$, 
and if this is $0$ then the $e_r$ coordinate of $e_r^\infty$ and $(e'_r)^\infty$ are not equal, that is, $b_{rr}^\infty-(b'_{rr})^\infty \neq 0$. But then again, as in the previous case we have 
\[\tilde{p}_{S_{u_1}}^\infty-(\tilde{p}'_{S_{u_1}})^\infty=p_{S_{u_1}}^\infty-(p'_{S_{u_1}})^\infty+(b_{rr}^\infty-(b'_{rr})^\infty)e_{u_1}\wedge \bigwedge_{i=1}^{\t_v}b_{ii}^\infty e_i \neq 0.\]
\end{itemize}
In any case, $\tilde{p}^\infty$ and $(\tilde{p}')^\infty$ differ in at least one term, proving injectivity of $\tilde{\rho}$. 
 
Next we prove (b)  and (c) of Proposition \ref{crucial} simultaneously. Assume that $r\in \bu$ or $\t_r \ge u_1$. The problem with this case is that $b_{rr}^\infty:=\lim_{m\to \infty}b^{(m)}_{rr} \in \CC \setminus \{0\}$ does not  necessarily hold any more and the limit can be $\infty$ or $0$. In both cases $\tilde{b}^{(m)}$ is ill-defined in \eqref{modifiedsl}.

Fix a nonzero $\delta \in \CC$ and define the sequence 
\begin{equation}\label{modifiedsldelta} \tilde{b}_{ij}^{(m),\delta}=\begin{cases} \delta  & (i,j)=(u_1,u_1) \\ \frac{1}{\delta}b^{(m)}_{rr}\cdot b^{(m)}_{u_1u_1} & (i,j)=(r,r) \\ \frac{1}{\delta}b_{rj}^{(m)}\cdot b_{u_1,u_1}^{(m)} & \text{ if } i=r, j> r \text{ and } \t_j \ge u_1 \\ b_{ij}^{(m)} & \text{otherwise} \end{cases}
\end{equation}
In short, we increase the diagonal entry $b_{u_1u_1}^{(m)}$ to be constant $\delta$ and multiply the entries in the $r$th row above the diagonal by $\frac{1}{\delta}b_{u_1u_1}^{(m)}$ whenever $\t_j \ge u_1$.
Then the new sequence still sits in $\SL_n(\CC)$. If $r\in \bu$ then $\mathrm{VSpec}(\tilde{b}^{(m),\delta})=\{u_2,\ldots, u_s\}$. If $r\notin \bu$ and $\t_r \ge u_1$ then 
\[\mathrm{VSpec}(\tilde{b}^{(m),\delta})=\begin{cases} \{u_2,\ldots, u_s\}& \text{ if } \lim_{m\to \infty}b^{(m)}_{rr}b^{(m)}_{u_1u_1} \neq 0 \\  \{u_2,\ldots, u_s\}\cup \{r\} & \text{ if } \lim_{m\to \infty}b^{(m)}_{rr}b^{(m)}_{u_1u_1} = 0 \end{cases}.\]
In any case, if the limit exists then $\tilde{p}^{\infty,\delta}=\lim_{m\to \infty} \tilde{b}^{(m),\delta}p_{\tS} \in \calb_{u_2,\ldots, u_s}$ or $\tilde{p}^{\infty,\delta} \in \calb_{\{u_2,\ldots, u_s\}\cup \{r\}}$.

 The same argument as for part (a) shows that 
\begin{itemize}
\item If $V \in \tS$ with $\t_V \ge u_1$ then $\tilde{p}^{\infty,\delta}_V=p^\infty_V$. 
\item If $V \in \tS$ with $\t_V < u_1$ then $\lim_{m\to \infty}b^{(m)}e_v\in \CC^n$ and $\lim_{m\to \infty}\tilde{b}^{(m),\delta}e_v\in \CC^n$ exists for all $v\in V$ by Lemma \ref{firstmu} and therefore 
\[p^\infty_V=\lim_{m\to \infty}\wedge_{v\in V}b^{(m)}e_v=\wedge_{v\in V} \lim_{m\to \infty}b^{(m)}e_v\]
and the same holds with $b^{(m)}$ replaced by $\tilde{b}^{(m),\delta}$. 
But
\[\lim_{m\to \infty}\tilde{b}^{(m),\delta}e_v=\begin{cases} \lim_{m\to \infty}b^{(m)}e_v & v\neq u_1 \\ \lim_{m\to \infty}b^{(m)}e_{u_1}+\delta e_{u_1} & v=u_1 \end{cases}.\]
In particular, when $\delta \to 0$ the point $\tilde{p}^{\infty,\delta}_V$ tends to $p^\infty_V$. 
\end{itemize}
In short,  
\begin{equation}\label{tends}
\lim_{\delta \to 0} \tilde{p}^{\infty,\delta}=p^\infty
\end{equation} 
and therefore
\[p^{\infty} \in \begin{cases} \overline{\calb_{\{u_2,\ldots, u_s\}}} & \text{ if } r\in \bu \\ \overline{\calb_{\{u_2,\ldots, u_s\}}} \text{ or } \overline{\calb_{\{u_2,\ldots, u_s\} \cup \{r\}}} & \text{ if } \t_r \ge u_1\end{cases} . \]
But $p^\infty \in \calb_{\{u_1,\ldots, u_s\}}$ so $p^{\infty}\notin \calb_{(u_2,\ldots, u_s)}$ and $p^{\infty}\notin \calb_{\{u_2,\ldots, u_s\}\cup \{r\}}$ and we are done.  

\end{proof}

We are ready to finish the proof of Theorem \ref{main} for $G=\SL_n(\CC)$. The key observation is that in Proposition \ref{crucial} the smallest element of the vanishing spectrum $\{u_2,\ldots ,u_s\}\cup \{r\}$--which is either $r$ or $u_2$--is strictly bigger than $u_1$. We call a sequence $\mathbf{r}=\{r_0<r_1<\ldots <r_t\}$ compatible with $\bu^0=\{u_1^0<\ldots <u_s^0\}$ if 
\[r_i \notin \bu^i, \t_{r_i}<\bu^i_1 \text{ for } i=1, \ldots, t-1 \text { but } r_t \in \bu^t \text{ or } \t_{r_t}\ge \bu^t_1 \]
where $\bu^1,\ldots, \bu^{t-1}$ are defined inductively as
\[\bu^i=\bu^{i-1} \setminus \{\min(\bu^{i-1})\} \cup r_{i-1}.\]
Here $\min(\bu)$ is the minimal element of the set $\bu$. For a compatible sequence $\br$ Proposition \ref{crucial} gives a sequence of injective maps  
\[\rho_i: \calb_{\bu^{i-1}}^{r_{i-1}} \hookrightarrow \calb_{\bu^i} \text{ for } i=1,\ldots, t.\]
We define by induction the subsets
\[\calb_{\bu^{i-1}}^{r_{i-1},\ldots, r_t}:=\rho_i^{-1}(\calb_{\bu^{i}}^{r_{i},\ldots, r_t}) \text{ for } 1\le i \le t.\]
Then Proposition \ref{crucial} gives a chain of injections
\begin{equation}\label{chain}
\calb_{\bu^0}^{r_0,\ldots, r_t} \hookrightarrow \calb_{\bu^1}^{r_1,\ldots, r_t} \hookrightarrow \ldots \hookrightarrow\calb_{\bu^{t-1}}^{r_{t-1},r_t} \hookrightarrow \calb_{\bu^t}^{r_t}.
\end{equation}
For a covered sequence $\bu$ let $\Gamma_{\bu}$ denote the set of $\bu$-compatible sequences. Then  
\[\calb_\bu=\bigcup_{\br \in \Gamma_{\bu}}\calb_{\bu}^{\br}\]
so it is enough to prove that the codimension of each $\calb_{\bu}^{\br}$ in $\overline{B\cdot p_S}$ is at least two. 
This is clear if $|\bu_t|\ge 2$. Indeed, by definition $r_t \in \bu^t$ or $\t_{r_t} \ge \bu^t_1$ holds and therefore by Proposition \ref{crucial} (b) and (c) \eqref{chain} can be extended with one of the embeddings
\[\calb_{\bu^t}^{r_t} \hookrightarrow \overline{\calb_{\bu^{t}\setminus \{\min(\bu^t)\}}} \setminus \calb_{\bu^{t}\setminus \{\min(\bu^t)\}} \text{ or } \calb_{\bu^t}^{r_t} \hookrightarrow \overline{\calb_{\bu^{t}\setminus \{\min(\bu^t)\}\cup \{r_t\}}} \setminus \calb_{\bu^{t}\setminus \{\min(\bu^t)\}\cup \{r_t\}}.\]
 So $\calb_{\bu^t}^{r_t}$ sits in the boundary of a boundary component and therefore has codimension at least $2$ in $\overline{B\cdot p_{\tS}}$. So Theorem \ref{main} is reduced to the special case when $\bu=\{u\}$ has a single element. In order to handle this case need one more definition.   




\begin{definition}\label{ivfixed} Let $1\le i \le n-1$ an integer and $v\in \CC^n$ such that $v\notin  \mathrm{Span}(e_1,\ldots, e_i)$. A point $p=\bigoplus_{U\in \tS} p_U \in \calw_{\tS}$ is called $(i,v)$-fixed if the stabiliser $G_p \subset \SL_n(\CC)$ of $p$ contains the one parameter subgroup  
\[T^{i,v}(\lambda): e_j \mapsto \begin{cases} e_j & j\neq i \\ e_i+\lambda v & j=i \end{cases} \text{ for } \lambda\in \CC.\]
A subset $\calb \subset \overline{B_n\cdot p_{\tS}}$ is called $i$-fixed if every point of $\calb$ is $(i,v)$-fixed for some $v\notin  \mathrm{Span}(e_1,\ldots, e_i)$. 
\end{definition}

 


\begin{lemma}\label{fundamentalcor}
Let $\calb \subset \overline{B_n \cdot p_{\tilde{S}}}$ be an $i$-fixed Borel invariant subvariety for some $1\le i \le n-1$. Then 
\[\dim \overline{\SL_n(\CC)\cdot \calb} \le \dim \overline{\SL_n(\CC)\cdot p_{\tilde{S}}}-2.\]
\end{lemma}

\begin{proof}
Consider the map 
\[\varphi: \SL_n(\CC) \times \calb \to \calw_{\tS},\ \ \ \varphi(g,w) \mapsto g \cdot w.\]
Choose $w\in \calb$ and let $T^{i,w}(\l)$ be the corresponding 1-parameter subgroup as in Definition \ref{ivfixed}. Since $\calb \subset \calw_{\tS}$ is Borel-invariant, the fibre $\varphi^{-1}(g\cdot w)$ contains $(g(bT^{i,w}(\lambda))^{-1},(bT^{i,w}(\lambda))\cdot w)$ for $b\in B_n, \lambda\in \CC$. Since $\{T^{i,w}(\l):\l \in \CC\} \cap B_{n}=\{1\}$, 
\[\dim(\{bT^{i,w}(\lambda):b\in B_{n},\l \in \CC\}=\dim(B_{n})+1\] 
and we get    
\begin{multline}\nonumber
\dim(\im(\varphi))=\dim \SL_n(\CC)+\dim \calb-\dim(\mathrm{fibre})\le \dim \SL_n(\CC)+\dim\overline{B_n \cdot p_{\tS}}-1-(\dim(B_n)+1)=\\
=\dim \SL_n(\CC)/U_S -2=\dim \overline{\SL_n(\CC)\cdot p_{\tilde{S}}}-2.
\end{multline}

\end{proof}

\begin{lemma}\label{lemma:ufixed} Let $\bu=\{u\}$ has one element and $r>u$.  Then $\calb_{u}^r$ is $u$-fixed and therefore by Lemma \ref{fundamentalcor}
\[\dim \overline{\SL_n(\CC) \calb_u^r} \le \dim \overline{\SL_n(\CC)p_S}-2.\]
\end{lemma}

\begin{proof}
Let $p^\infty \in \calb^r_{u}$ and $Z\in \calv(p^\infty)_{\min}$. By Proposition \ref{firstprop}
\begin{enumerate}
\item $[p^{\infty}_Z] \subset \bigcap_{V\in \calv(p^\infty)}p_V^\infty$ where $\calv(p^\infty)=\{U \in \tS: p^{\infty}_U\neq 0, \t_U \ge u\}$. 
\item $\omega(p^\infty)=r$, and hence there is a vector $w=e_r+w_{r-1}e_{r-1}+\ldots +w_1e_1 \in [p_Z^\infty]$. 
\end{enumerate}
By Lemma \ref{firstmu}
\[e_j^\infty=\lim_{m\to \infty}b^{(m)}e_j=\mu_{jj}e_j+\ldots +\mu_{1j}e_1 \text{ exists when } \t_j<u\]
and in particular, since only $b^{(m)}_{uu}$ tends to zero among the diagonal entries,  we have 
\[\mu_{jj} \neq 0 \text{ when } u<j, \t_j <u.\]
Therefore the linear base change 
\[A: \tilde{e}_j := \begin{cases} e_j^\infty & u<j, \t_j <u \\ w & j=r \\ e_j & \text{ otherwise} \end{cases} \]
sits in the Borel $B_{\SL_n}$. Since $\{1,\ldots, u\} \subset Z$, by (i) and (ii) above we have
\begin{equation}\label{pvinftysl2}
\mathrm{Span}(\tilde{e}_1,\ldots, \tilde{e}_{u},\tilde{e}_r)\subseteq [p_Z^\infty] \subseteq [p_V^\infty] \text{ for all } V\in \tS \text{ with } p_V^\infty \neq 0, \t_V\ge u.
\end{equation} 
If $\t^F_V < u$ then by Lemma \ref{firstmu} again $p_V^\infty=\wedge_{v\in V}e_v^\infty=\wedge_{v\in V}\tilde{e}_v$. But $\lim_{m\to \infty} b^{(m)}_{uu}=0$ and hence 
\[\lim_{m\to \infty} b^{(m)}e_u \subseteq \mathrm{Span}(e_1,\ldots, e_{u-1})=\mathrm{Span}(\tilde{e}_1,\ldots, \tilde{e}_{u-1})\]
and therefore
\begin{equation}\label{pvinftysl1}
[p_V^\infty] \subset \mathrm{Span}(\tilde{e}_1,\ldots, \tilde{e}_{u-1}, \tilde{e}_{u+1}, \ldots, \tilde{e}_n) \text{ for all } V\in \tS \text{ with }  \t_V < u.
\end{equation} 
From \eqref{pvinftysl1} and \eqref{pvinftysl2} follows that the one parameter subgroup 
\[\tilde{T}^{u,\tilde{e}_r}(\lambda): \tilde{e}_j \mapsto \begin{cases} \tilde{e}_j & j\neq u \\ \tilde{e}_u+\lambda \tilde{e}_r & j=u \end{cases} \text{ for } \lambda\in \CC.\]
stabilises $p^\infty$ so $T^{u,A^{-1}\tilde{e}_r}$ stabilises $p^\infty$ in the old basis, proving that it is $u$-fixed.  
\end{proof}



\section{Borel-regular subgroups of classical groups}\label{sec:classical}
In this section, again, we restrict our attention to the $k=\CC$ case, but all arguments work for any algebraically closed field $k$ of characteristic zero which is a subfield of $\CC$. We will often use the shorthand notation $\Sp_n$ for $\Sp_n(\CC)$ and $\SO_n$ for $\SO_n(\CC)$.  

Recall from the Introduction that a Borel-regular subgroup of a linear algebraic group is a subgroup normalized by a Borel subgroup. They have the form $U_S$ corresponding to closed root subsets $S \subset R^+$ which are also closed under shifting by elements of $R^+$, i.e. $S+r \subset S$ for any $r\in R^+$. 

\subsection{Borel-regular subgroups of $\SL_n$} When $G=\SL_n$ this means that $(i,j)\in S \Rightarrow (i,j+1),(i-1,j)\in S$, hence Borel-regular subgroups have the form 
\begin{equation}\label{block}U^{0,0,1,2,2,3}=\begin{array}{|cccccc|}
\hline 
1 & 0 & \bullet & \bullet & \bullet & \bullet \\
  & 1 & 0 & \bullet & \bullet & \bullet \\
  &   & 1 & 0 & 0 & \bullet \\
  &   &   & 1 & 0 & 0\\
  & & & & 1 & 0 \\
  & & & & & 1\\
\hline
\end{array}
\end{equation}
where the positions of free parameters are encoded by a monotone increasing sequence $\theta=(\theta_1 \le \ldots \le \theta_n)$  satisfying $\t_i<i$. This sequence then corresponds to the root subset $S=\{S_1,\ldots, S_n\}$ where $S_i=\{1,\ldots, \t_i,i\}$. Note that a subgroup of $\SL_n$ is Borel regular if and only if it is left and right Borel regular at the same time. Therefore the Popov-Pommerening conjecture for Borel-regular subgroups of $\SL_n$ is a special case of Theorem \ref{main}.

\subsection{Borel-regular subgroups of symplectic groups}\label{subsec:symplectic}

Let $V$ be a $n=2l$-dimensional complex vector space and $Q:V \times V \to \CC$ a non-degenerate skew-symmetric bilinear form on $V$. The symplectic Lie group is then 
\[\Sp_{n}(V)=\{A \in \SL_{n}(\CC):Q(Av,Aw)=Q(v,w) \text{ for all } v,w\in V\},\]
and the corresponding symplectic Lie algebra is 
\[\spc_{n}(V)=\{A \in \slc_{n}(\CC):Q(Av,w)+Q(v,Aw)=0 \text{ for all } v,w\in V\}.\]
To get a compatible embedding of $\Sp_n(\CC)\subset \SL_n(\CC)$ with diagonal maximal torus and which preserves the standard Borel of upper triangular matrices in $\SL_n$,  we take a basis $e_1,\ldots, e_{n}$ of $V$ such that $Q$ is given by the matrix $M$ in the form $Q(v,w)=v^tMw$ where $M$ is the antidiagonal $n\times n$ matrix with two antidiagonal $l \times l$ blocks: 
\[M={\tiny \left( \begin{array}{cccccc} 
 &  & &&&1 \\
 &&& & \iddots &  \\
&&& 1 &  &  \\
&& -1 &&&\\
& \iddots &&&&\\
-1&&&&&
\end{array} \right)}\]
For a diagonal matrix $D=\mathrm{diag}(t_1,\ldots, t_n)$, the condition to lie in $\Sp_n$ is $DMD=M$. Since
\[DMD={\tiny \left( \begin{array}{ccccc} 
 &  & & & t_1t_n \\
 &  & & t_2t_{n-1} &  \\
 & & \iddots & & \\
 & -t_{n-1}t_2 &  & & \\
 -t_nt_1 & & & & 
\end{array} \right)}\]
this happens exactly when $t_1t_n=\ldots =t_lt_{l+1}=1$. Hence the maximal torus in $\Sp_n$ is 
\[T_{\Sp_n}={\tiny \left\{\left( \begin{array}{cccccc} 
 t_1 &  & & & & \\
 & \ddots  & &  & & \\
 & & t_l & & & \\
 &  &  &  t_l^{-1} & & \\
 & & & & \ddots &\\
 &&&&& t_1^{-1} 
\end{array} \right):t_1,\ldots, t_l \in \CC^*\right\}}\]
and the rank of $\Sp_{2l}$ is $l$. For $1\le i \le l$ define the character $\a_i:T_{\Sp_n} \to \CC^*$ by 
\[\a_i(\diag(t_1,t_2, \ldots, t_2^{-1},t_1^{-1}))=t_i\] and the cocharacter $\lambda_i:\CC^* \to T_{\Sp_n}$ by $\l_i(x)=\diag(1,\ldots, x,\ldots ,x^{-1},\ldots , 1)$ with $x$ at the $i$th position. Then $X^*(T_{\Sp_n})=\oplus_{i=1}^l \ZZ \a_i$ and $X_*(T_{\Sp_n})=\oplus_{i=1}^l \ZZ \l_i$ with dual pairing $\langle \a_i,\l_j \rangle=\delta_{ij}$.

For an $n \times n$ matrix $A=(a_{ij})$ let $A^{at}=(a^{at}_{ij})$ denote its antidiagonal-transpose, that is $a^{at}_{ij}=a_{n-j+1,n-i+1}$.
Computing $A^tM +MA=0$ shows that the Lie algebra $\spc_n$ consists of matrices of the form
\[\left( \begin{array}{cc} 
 A & B \\
 C & -A^{at}\\
\end{array} \right)\]
where $B=B^{at}, C=C^{at}$. 
\begin{rem}\label{remark:liealgebra}
In particular this means that any Lie algebra element $A \in \spc_n$ is uniquely determined by its entries $\{a_{ij}:i+j\le n+1\}$ sitting above and on the antidiagonal. 
\end{rem}
The Cartan subalgebra $\lieh \subset \spc_{n}$ is $l$-dimensional, spanned by the diagonal matrices $E_{ii}-E_{n+1-i,n+1-i}$ whose dual is $\a_i$. Here, as before, $E_{ii}$ is the matrix with $1$ in the diagonal entry $(i,i)$ and zero elsewhere. The roots of $\Sp_n$ are
\[R=\{\pm \a_i \pm \a_j\}_{1\le i,j \le l}\]
and the positive roots are 
\[R^+=\{\a_i-\a_j\}_{i<j}\cup \{\a_i+\a_j\}_{i\le j}.\]
Therefore the Lie algebra of the corresponding Borel subgroup consists of matrices of the form 
\[\left( \begin{array}{cc} 
 A & B \\
 0 & -A^{at}\\
\end{array} \right)\]
where $A$ is upper triangular and $B=B^{at}$. 

All one dimensional positive root subspaces $\lieg_\alpha$ have the form 
\begin{equation}\label{rootsubspace}
\lieg_\alpha=\left\{\left(\begin{array}{cccc} 
 0 &  & x &  \\
   & 0 &  & (x \text{ or } -x)\\
   & & 0 & \\
   & & & 0
\end{array} \right):x\in \CC \right\}
\end{equation}
where 
\[x \text{ sits at }  (i,j) \text{ and } x \text{ sits at } (n+1-j,n+1-i) \text{ if } \alpha=\a_i-\a_j\] 
\[x \text{ sits at } (i,j+l) \text{ and } (-x) \text{ sits at } (l+1-j,n+1-i) \text{ if } \a=\a_i+\a_j, i\neq j\]
\[x \text{ sits at }  (i,n+1-i) \text{ if } \alpha=2\a_i.\]
The corresponding root subgroups $U_\alpha=\exp(\lieg_\alpha)$ have the same form with $1$'s on the diagonal. Figure \ref{figuretwo} shows the positive root spaces and root subgroups for $n=4$. 

\begin{figure}[h]
\centering
{\small \begin{tabular}{| c | c | c |}
$\alpha$ & $\lieg_\alpha$ & $U_\alpha$ \\
\hline			
$2\alpha_1$ & $\left(\begin{array}{cccc} 
 0 & & & x \\
   & 0 & & \\
   & & 0 &\\
   & & & 0
\end{array} \right)$ & $\left(\begin{array}{cccc} 
 1 & & & x \\
   & 1 & & \\
   & & 1 &\\
   & & & 1
\end{array} \right)$ \\
  $2\a_2$ & $\left(\begin{array}{cccc} 
 0 & & &  \\
   & 0 & x & \\
   & & 0 &\\
   & & & 0
\end{array} \right)$ & $\left(\begin{array}{cccc} 
 1 & & &  \\
   & 1 & x & \\
   & & 1 &\\
   & & & 1
\end{array} \right)$ \\
  $\a_1-\a_2$  & $\left(\begin{array}{cccc} 
 0 & x & &  \\
   & 0 &  & \\
   & & 0 & -x\\
   & & & 0
\end{array} \right)$ & $\left(\begin{array}{cccc} 
 1 & x & &  \\
   & 1 &  & \\
   & & 1 & -x\\
   & & & 1
\end{array} \right)$  \\  
$\a_1+\a_2$  & $\left(\begin{array}{cccc} 
 0 &  & x &  \\
   & 0 &  & x\\
   & & 0 & \\
   & & & 0
\end{array} \right)$ & $\left(\begin{array}{cccc} 
 1 &  & x&  \\
   & 1 &  & x \\
   & & 1 & \\
   & & & 1
\end{array} \right)$ \\
\hline 
\end{tabular}}
\caption{Root subspaces and root subgroups of $\Sp_4$.}
\label{figuretwo}
\end{figure}
\begin{definition}\label{def:sisymplectic} For a closed subset $S\subset R^+$ let $U_S^{\Sp}=\langle U_{\alpha}:\alpha \in S\rangle \subset \Sp_{n}$ be the corresponding unipotent subgroup generated by the root subgroups in $\Sp_{n}$, normalized by the maximal diagonal torus in $\Sp_{n}$.
We define the family $S=\{S_1,\ldots, S_n\}$ of subsets of $\{1,\ldots, n\}$ in such way that $S_i$ collects all possible non-zero entries in the $j$th column in $U_S^{\Sp} \subset \Sp_n \subset \SL_n$, that is
\[S_j=\{i: \exists u\in U_S^{\Sp}\subset \Sp_n \subset \SL_n \text{ such that } u_{ij}\neq 0\}\]
These subsets can be described using the roots in $S$ as follows:
\begin{equation}\nonumber
S_j=\begin{cases} \{j\} \cup \{i: \a_i-\a_{j} \in S\} & \text{ for } 1\le j \le l \\
\{j\}\cup  \{i: \a_i+\a_{j-l} \text{ or } \a_{n+1-j}+\a_{l+1-i} \text{ or } \a_{n+1-j}-\a_{n+1-i} \text{ is in } S & \text{ for } l+1\le j \le 2l.
\end{cases}
\end{equation}
\end{definition}

In what follows, we will use the same notation $S$ for a set of positive roots for $\Sp_n$, for the corresponding subset family $S=\{S_1,\ldots, S_n\}$ and the corresponding set of entries 
\[S=\{(i,j) \in \{1,\ldots, n\} \times \{1,\ldots, n\}:i\in S_j\}\]
indexing all possibly nonzero entries of $U_S^{\Sp} \subset \Sp_n \subset \SL_n$.
\begin{exit} For $\Sp_4$ the closed subset $\{\a_1-\a_2,\a_1+\a_2,2\a_1\}\subset R^+$ defines the regular subgroup  
\[U_{S}^{\Sp}=\left\{\left( \begin{array}{cccc} 
1 & a & b & c \\
0 & 1 & 0 & b \\
0 & 0 & 1 & -a \\
0 & 0 & 0 & 1
\end{array} \right):a,b,c\in \CC\right\}\subset \Sp_4(\CC)\]
with the corresponding subset family    
\[S_1=\{1\},S_2=\{1,2\},S_3=\{1,3\},S_4=\{1,2,3,4\}.\]
\end{exit}

\begin{lemma}\label{lemma:ts} \begin{enumerate}
\item If $U_S^{\Sp} \subset \Sp_n \subset \SL_n$ is a Borel-regular subgroup then the subset family $S=\{S_1,\ldots, S_n\}$ defines a Borel regular subgroup $U_S^{\SL}$ of $\SL_n$ such that $U_S^{\Sp}=U_S^{\SL} \cap \Sp_n$.  Equivalently, if $(i,j) \in S$ then $(i-1,j),(i,j+1) \in S$.
\item $S$ is symmetric about the antidiagonal: $(i,j)\in S \Leftrightarrow (n+1-j,n+1-i) \in S$.
\end{enumerate}
\end{lemma}

\begin{proof}
(i) follows from the fact that the embedding $\Sp_n \subset \SL_n$ preserves the Borel subgroup of upper triangular matrices in $\SL_n$. Positive root subspaces of $\SL_n$ correspond to entries above the diagonal, and every such entry defines a unique positive root subspace of $\Sp_n$.   

(ii) follows from the symmetry of the Lie algebra $\spc_n$ in $\mathfrak{sl}_n$: all root subspaces are symmetric about the antidiagonal in $\mathfrak{sl}_n$.
\end{proof}

\begin{exit} In $\Sp_4$ the subset $S=\{2\a_1,2\a_2,\a_1+\a_2\}\subset R^+$ is closed under addition of positive roots and therefore defines a Borel-regular subgroup  
\[U_{S}^{\Sp}=\left\{\left( \begin{array}{cccc} 
1 & 0 & a & b \\
0 & 1 & c & -a \\
0 & 0 & 1 & 0 \\
0 & 0 & 0 & 1
\end{array} \right):a,b,c\in \CC\right\}\subset \Sp_4(\CC).\]
The corresponding subset family is    
\[S_1=\{1\},S_2=\{2\},S_3=\{1,2,3\},S_4=\{1,2,4\}\]
which defines the Borel-regular subgroup
\[U_{S}^{\SL}=\left\{\left( \begin{array}{cccc} 
1 & 0 & a & b \\
0 & 1 & c & d \\
0 & 0 & 1 & 0 \\
0 & 0 & 0 & 1
\end{array} \right):a,b,c,d \in \CC\right\}\subset \SL_4.\] 
\end{exit}

Lemma \ref{lemma:ts} implies that $U_S^{\SL}=U^{\t} \subset \SL_n$ corresponds to some monotone increasing sequence $\t=(\t_1\le \ldots \le \t_n)$ where $S_i=\{1,\ldots, \t_i,i\}$. Due to the antidiagonal symmetry there is a unique integer $1 \le \g_S \le l$ which satisfies that 
\[(\gamma_S,n+1-\gamma_S)\in S \text{ but } (\gamma_S+1,n-\gamma_S) \notin S\]
We call $\gamma_S$ the \textbf{crossing point} of $S$ because the boundary of the region of free parameters intersect the antidiagonal at the point $(\g_S,\g_S)$, see Figure \ref{fig:funddomain} for an example. 

\begin{definition}\label{def:fundamental}
Let $S=(S_1,\ldots, S_n) \subset \{1,\ldots, n\} \times \{1,\ldots, n\}$ be a domain which is 
\begin{enumerate}
\item symmetric about the antidiagonal, that is $(i,j) \in \cals \Leftrightarrow (n+1-j,n+1-i) \in \cals$
\item Borel-regular, that is, $(i,j) \in \cals \Rightarrow (i,j+1),(i-1,j) \in \cals$. 
\end{enumerate}
We define the \textbf{symplectic fundamental domain} to be $F=\{F_1,\ldots, F_n\}$ such that 
\[F_i=\begin{cases} S_i & i\le n-\gamma_\cals \\ \{1,2,\ldots, i\} & i> n-\gamma_\cals
\end{cases}\] 

\begin{figure}[h]
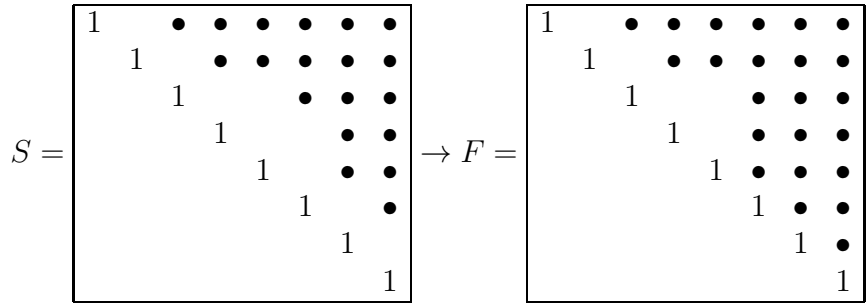

\centering
\[S=\begin{array}{|cccccccc|}
\hline 
1 &  & \bullet & \bullet & \bullet & \bullet & \bullet & \bullet \\
  & 1 &  & \bullet  & \bullet & \bullet & \bullet & \bullet \\
  &   & 1 &    &  & \bullet & \bullet & \bullet \\
  &   &   & 1  &   &  & \bullet & \bullet\\
  &  &   &   & 1 & & \bullet & \bullet\\
  & & & & & 1& & \bullet\\
  & & & & & & 1& \\ 
  & & & & & & &1\\
\hline
\end{array} \to F=\begin{array}{|cccccccc|}
\hline 
1 &  & \bullet & \bullet & \bullet & \bullet & \bullet & \bullet \\
  & 1 &  & \bullet  & \bullet & \bullet & \bullet & \bullet \\
  &   & 1 &   &  & \bullet & \bullet & \bullet \\
  &   &   & 1  &   & \bullet & \bullet & \bullet\\
  &  &   &   & 1 & \bullet & \bullet & \bullet\\
  & & & &  & 1& \bullet & \bullet\\
  & & & & &  & 1& \bullet \\ 
  & & & &  & &  &1\\
\hline
\end{array}\] 
\caption{A domain $S$ symmetric about the antidiagonal with crossing point $\g_S=3$ and its symplectic fundamental domain}
\label{fig:funddomain}
\end{figure}
\end{definition} 

Note that $F$ is no longer symmetric about the anti-diagonal but it remains Borel regular and $U_F=U^{\t_F}$ corresponds to the modified sequence 
\[\t_F=\{\t_1,\ldots, \t_{n-\g_S},n-\g_S,\ldots, n-1\}.\]
See Figure \ref{fig:funddomain} for an example. 
Recall that we fixed a basis $\{e_1,\ldots, e_n\}$ of $\CC^n$ to get the embedding $\Sp_n \subset \SL_n$ of the right form. The corresponding point 
\[p_{F}=\bigoplus_{U \in F} \wedge_{i\in U} e_i \in 
\calw_{F}=\bigoplus_{U \in F} \wedge^{|U|}\CC^n.\]
has the right stabiliser in $\Sp_n$ according to the following
\begin{theorem}
The stabilizer of $p_{F}$ in $\Sp_n$ is $U_S^{\Sp}$.
\end{theorem}

\begin{proof}
Let $\mathrm{Stab}_{\Sp_n}(p_F)$ (resp. $\mathrm{Stab}_{\SL_n}(p_F)$) denote the stabiliser of $p_F$ in $\Sp_n$ (resp. $\SL_n$). Then 
\[\mathrm{Stab}_{\Sp_n}(p_F)=\mathrm{Stab}_{\SL_n}(p_F) \cap \Sp_n.\]
But according to Lemma \ref{stabiliser}, $\mathrm{Stab}_{\SL_n}(p_F)=U_F^{\SL}$ and due to the antidiagonal symmetry of $\Sp_n$,  $U_F^{\SL} \cap \Sp_n=U_S^{\Sp}$ which proves the statement.  
\end{proof}

We define the symplectic  Grosshans pair using  the corresponding ring family as before:  
\[\tF=\{F_{i_1}\cup \ldots \cup F_{i_r}:1\le i_1<\ldots <i_r\le n\}\subset 2^{\{1,\ldots ,n\}}\]
and then 
\[p_{\tF}=\bigoplus_{U \in \tF} \wedge_{i\in U} e_i \in 
\calw_{\tF}=\bigoplus_{U \in \tF} \wedge^{|U|}\CC^n.\]
has the same stabiliser in $\Sp_n$ as $p_F$.

\begin{corollary}
The stabilizer of $p_{\tF}$ in $\Sp_n$ is $U_S^{\Sp}$.
\end{corollary}

\begin{theorem}\label{prop:symplectic}
Let $n=2l$ and $F=\{F_1,\ldots, F_n\}$ be the symplectic fundamental domain corresponding to a Borel-regular subgroup $U_S^{\Sp} \subset \Sp_n$ with $\g_S=l$. Then the pair $(\calw_{\tF},p_{\tF})$ is a Grosshans pair for $U_S^{\Sp}$. This proves Theorem \ref{main2} for symplectic groups. 
\end{theorem}

\begin{proof}
$\g_S=l$ is equivalent to saying that the top right quarter  $\{(i,j): 1\le i \le l, l+1\le j \le n \}$ belongs to $F$.  In other words $F$ corresponds to a sequence 
\begin{equation}\label{thetaf}
\t^F=(\t_1\le \ldots \le \t_l, l,l+1,\ldots ,n-1).
\end{equation}
See Figure \ref{fig:fatfunddomain} for an example.

\begin{figure}[h!]
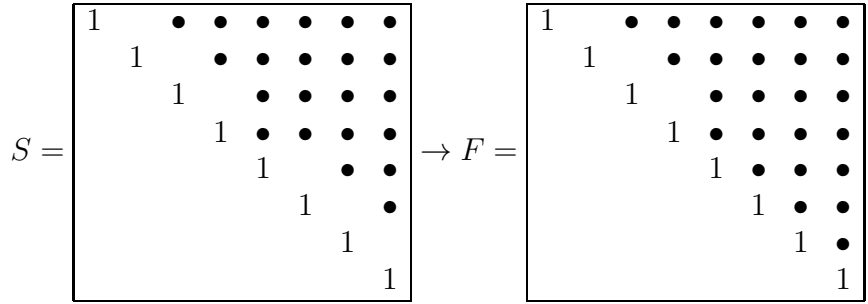

\centering
\[S=\begin{array}{|cccccccc|}
\hline 
1 &  & \bullet & \bullet & \bullet & \bullet & \bullet & \bullet \\
  & 1 &  & \bullet  & \bullet & \bullet & \bullet & \bullet \\
  &   & 1 &    & \bullet & \bullet & \bullet & \bullet \\
  &   &   & 1  &  \bullet  & \bullet & \bullet & \bullet\\
  &  &   &   & 1 & & \bullet  & \bullet\\
  & & & & & 1& & \bullet\\
  & & & & & & 1& \\ 
  & & & & & & &1\\
\hline
\end{array} \to F=\begin{array}{|cccccccc|}
\hline 
1 &  & \bullet & \bullet & \bullet & \bullet & \bullet & \bullet \\
  & 1 &  & \bullet  & \bullet & \bullet & \bullet & \bullet \\
  &   & 1 &   & \bullet  & \bullet & \bullet & \bullet \\
  &   &   & 1  & \bullet  & \bullet & \bullet & \bullet\\
  &  &   &   & 1 & \bullet & \bullet & \bullet\\
  & & & &  & 1& \bullet & \bullet\\
  & & & & &  & 1& \bullet \\ 
  & & & &  & &  &1\\
\hline
\end{array}\]
\caption{The fundamental domain of a fat Borel-regular subgroup of $\Sp_8$. Here $\t_S=(0,0,1,2,4,4,5,6)$ and $\t_F=(0,0,1,2,4,5,6,7)$.}
\label{fig:fatfunddomain}
\end{figure}

Recall that elements of the Borel $B_{\Sp_n} \subset \Sp_n$ have the form
\[\left( \begin{array}{cccc}
b_{11} & b_{12} & \cdots & b_{1n}\\
0 & b_{22} &   & \\
 & & \ddots & b_{n-1n}\\
 0 & & 0 & b_{nn}
\end{array}\right) \text{ where } b_{ii}=b_{n+1-i,n+1-i}^{-1}\]
and the off-diagonal entries $b_{ij}$ are not independent, but we will not use the exact form of these entries in this argument. The boundary points in $\overline{B_{\Sp_n} \cdot p_{\tF}}$ are limits of the form 
\[p^{\infty}=\lim_{m\to \infty}b^{(m)}\cdot p_{\tF}\]
where $(b^{m)})$ is a normalized sequence in the sense of Definition \ref{def:vanish}. We define the vanishing spectrum and the sets $B^{\Sp}_\bu$ exactly the same way as they are defined for $\SL_n$ in Definition \ref{def:vanish} and Definition \ref{def:bu} that is for $\bu \subset \{1,\ldots, n\}$ we let
\[\calb^\Sp_\bu=\{p^\infty \in \overline{B_{\Sp_n} \cdot p_{\tF}}: \exists \text{ norm. seq. } (b^{(m)})\subset \Sp_n \text{ s.t } p^\infty=\lim_{m\to \infty}b^{(m)} \cdot p_{\tF} \text{ and } \mathrm{VSpec}(b^{(m)})=\bu\}.\]
The embedding $\Sp_n \subset \SL_n$ implies that 
\[\calb^\Sp_\bu \subseteq \calb_\bu^{\SL}.\]

Note that the first proof of Lemma \ref{borelorbit} applies for the symplectic case and therefore 
\[\calb^\Sp_\emptyset=B_{\Sp_n} \cdot p_{\tF}.\]
This means that, again, all boundary points sit in a $B^\Sp_\bu$ with some nonempty $\bu$:
\[\overline{B_{\Sp_n} \cdot p_{\tF}}\setminus B_{\Sp_n} \cdot p_{\tF} \subset \cup_{\bu \neq \emptyset}\calb^\Sp_\bu.\]
In what follows we adapt the argument developed for left Borel regular subgroups of $\SL_n$ to Borel regular subgroups of $\Sp_n$. We start with two remarks on notations and definitions. 
\begin{enumerate}
\item  Since $\t^F$ is monotone increasing, $\theta^F_Z=\theta^F_{\max(Z)}$ holds for all $Z \subseteq \{1,\ldots, n\}$.  
\item  Since $\t^F_n=n$, every subset $\bu$ is automatically covered in the sense of Definition \ref{def:covered}.
\end{enumerate}
We have the following stronger version of Lemma \ref{mustbecovered}.
\begin{lemma}\label{lemma:u1} If $\bu=\{u_1< \ldots <u_s\} \subseteq \{1,\ldots, n\}$ is such that $u_1>l$ then $\calb^\Sp_\bu=\emptyset$.
\end{lemma}
\begin{proof} 
Let $(b^{(m)})$ be a normalized sequence with vanishing spectrum $\bu$. If $u_1\ge l+1$, or equivalently $n+1-u_1\le l$ then the relation $b_{ii}^{(m)}=(b_{n+1-i,n+1-i}^{(m)})^{-1}$ implies that 
\[\lim_{m\to \infty} b_{ii}^{(m)}\neq 0 \text{ for } 1\le i <n+1-u_1 \text{ and} \lim_{m\to \infty} |b_{n+1-u_1,n+1-u_1}^{(m)}|=\infty.\]
However, $\{1,\ldots, n+1-u_1\} \in \tF$ and 
\[p_{\{1,\ldots, n+1-u_1\}}^\infty=\lim_{m\to \infty} (\prod_{i=1}^{n+1-u_1}b^{(m)}_{ii})\cdot \bigwedge_{i=1}^{n+1-u_1} e_i\]
is not bounded, so this coordinate of $p^\infty$ does not exist,a contradiction.  
\end{proof}
Let $\bu=\{u_1<\ldots <u_s\}\subset \{1,\ldots, n\}$ be a subset with $u_1\le l$ and let 
\[p^\infty =\lim_{m\to \infty} b^{(m)} p_{\tF}=\oplus_{U\in \tilde{F}} p^\infty_{U} \in \calb^\Sp_\bu\]
 be a limit point defined by the normalized sequence $(b^{(m)}) \subset \Sp_n \subset \SL_n$ such that $\mathrm{VSpec}(b^{(m)})=\bu$. Note that $(b^{(m)})$ is a normalized sequence in the Borel $B_{\SL_n}$ of $\SL_n$ too, and therefore we can deduce the following. 
\begin{lemma}\label{firstmusymp} Lemma \ref{firstmu} and its proof remains valid when we replace $\calb_\bu$ with $\calb^\Sp_\bu$ and $\tS$ with $\tF$.
\end{lemma}
Definition \ref{def:important} works without change and Proposition \ref{firstprop} tells something about points in $\calb_\bu^{\SL}$ and since $\calb_\bu^\Sp \subset \calb_\bu^{\SL}$ these properties hold for points in $\calb_\bu^\Sp$ too. In short we have
 \begin{prop}\label{prop:remainstrue}
 Let $\bu \subset \{1,\ldots, n\}$ be a subset with $u_1\le l$. Then Propositon \ref{firstprop} holds for points in $\calb_\bu^{\Sp}$.
 \end{prop}
Using Definition \ref{def:bur} we can talk about the minimal set $\calv(p^\infty)_{\min}^{\Sp}$, the width of a boundary point $p^\infty$ and let  
\[\calb_{\bu}^{\Sp,r}=\{p^\infty \in \calb^{\Sp}_{\bu}: \omega(p^\infty)=r\}\]
denote the set of points in $\calb^{\Sp}_{\bu}$ of width $r$.
Then we have a (not necessary disjoint) finite decomposition
\[\calb^\Sp_{\bu}=\cup_{u_1 < r}\calb_{\bu}^{\Sp,r}.\] 
Remark \ref{remark:basechange} on unipotent base change remains valid if we choose our base change matrix $A$ from $B_{\Sp_n}$, and such a base change will leave $\calb_{\bu}^{\Sp,r}$ intact for all $\bu$ and $r$. As a corollary we have the following analog of Remark \ref{conditiononer}.

\begin{rem}\label{conditiononersymp} Let $p^\infty \in \calb^{\Sp,r}_\bu$ and $Z\in \calv(p^\infty)^{\Sp}_{\min}$. By definition this means that $[p^\infty_{Z}] \subset \mathrm{Span}(e_1,\ldots, e_r)$ but $[p^\infty_{Z}] \not\subset \mathrm{Span}(e_1,\ldots, e_{r-1})$ so there is a vector 
\[w=e_r+w_{r-1}e_{r-1}+\ldots +w_1e_1 \in [p^\infty_{Z}]\]
Define the Lie algebra element  
\[X^w:={\tiny \left( \begin{array}{cccccccccc}
0 &  & & w_1 & & & & & \\
   & \ddots & &  \vdots   & & & & & \\
   & & & w_{r-1} & & & & & \\
   & & & 0  & & & & & \\
   & & & & \ddots &  & & & \\
   & & & & & 0 & -w_{r-1} & \cdots & -w_1 \\
   & & & & &   &   \ddots  &   &    \\
   & & & & &   &   &   & 0 
\end{array}\right)}\in \spc_n\]
where we put $w_1,\ldots, w_{r-1}$ into the $r$th column and $-w_1,\ldots, -w_{r-1}$ into the $(n+1-r)$th row. 
$X^w$ sits in the Lie algebra of $B_{\Sp_n}$ and $\exp(X^w)$ is a unipotent element of $B_{\Sp_n}$. This defines the new basis $\bar{e}_i=\exp(X^w) \cdot e_i$. 
Let $p_{\tF}(\bar{e}_1,\ldots, \bar{e}_n)$ denote the base point $p_{\tF}$ written in the new basis. Then  
\[p^\infty=\lim_{m\to \infty} b^{(m)} p_{\tF}=\lim_{m\to \infty} (\exp(X^w)b^{(m)} \exp(-X^w)) p_{\tF}(\bar{e}_1,\ldots, \bar{e}_n).\] 
Similarly to Remark \ref{remark:basechange} we note that since $\exp(X^w)$ is unipotent, the new sequence 
\[\tilde{b}^{(m)}=\exp(X^w)b^{(m)} \exp(-X^w)\] 
has the same vanishing spectrum $\bu$ and $[p_Z^\infty]$ in this new basis has width $r$ again and therefore 
\[\calb_\bu^{\Sp,r}(e_1,\ldots e_n)=\calb_\bu^{\Sp,r}(\bar{e}_1,\ldots, \bar{e}_n) \text{ for all } \bu, r.\]
In this new basis, however,  $w=\tilde{e}_r \in [p^\infty_Z]$ holds. 
\end{rem}

The cornerstone of our argument for $\SL(n)$ was Proposition \ref{crucial} on the structure of the subsets $\calb_{\bu}^{r}$. The same proposition remains true for $\Sp_n$, but we need a careful review of the proof which was based on proper modifications of the sequence $(b^{(m)})$: the problem with the original argument is that the modified sequence sits in $\SL(n)$ but not necessarily in $\Sp_n$.
\begin{prop}\label{crucialsymplectic}
Let $s\ge 2$ and $\bu=\{u_1<\ldots <u_s\}$ be a subset with $u_1\le l$. Then    
\begin{enumerate}[(a)]
\item If $r \notin \bu$ and $\t^F_r <u_1$ then there is a continuous injection $\rho:\calb^{\Sp,r}_{\bu} \hookrightarrow \calb^{\Sp}_{\{u_2,\ldots, u_s\}\cup \{r\}}$ and therefore $\dim \calb^{\Sp,r}_{\bu} \le \dim \calb^\Sp_{\{u_2,\ldots, u_s\}\cup \{r\}}$.
\item If $r\in \bu$ then $\calb^{\Sp,r}_{\bu} \subset \overline{\calb^{\Sp,r}_{\{u_2,\ldots, u_s\}}}\setminus \calb^\Sp_{\{u_2,\ldots, u_s\}}$.  
\item If $r\notin \bu$ and $\t^F_r \ge u_1$ then $\calb^{\Sp,r}_{\bu} \subset \overline{\calb^\Sp_{\{u_2,\ldots, u_s\}}}\setminus \calb^\Sp_{\{u_2,\ldots, u_s\}}$ or $\calb^{\Sp,r}_{\bu} \subset \overline{\calb^\Sp_{\{u_2,\ldots, u_s\}\cup \{r\}}}\setminus \calb^\Sp_{\{u_2,\ldots, u_s\}\cup \{r\}}$.
\end{enumerate}
\end{prop}
 
 \begin{proof} To prove (a) assume that $r \notin \bu$ and $\t^F_r <u_1$. Let 
\[p^\infty =\lim_{m\to \infty} b^{(m)} p_{\tF}=\oplus_{U\in \tilde{F}} p^\infty_{U} \in \calb^{\Sp,r}_\bu\]
 be a limit point such that $\mathrm{VSpec}(b^{(m)})=\bu$ and $Z \in \calv(p^\infty)^{\Sp}_{\min}$. According to Remark \ref{conditiononersymp} we can assume that $e_r \in [p^\infty_{Z}]$.
 
By Proposition \ref{prop:remainstrue}, $\lim_{m\to \infty}b^{(m)}_{ii}\in \CC$ exists whenever $\t^F_i <u_1$ and since $r \notin \bu$, this limit is nonzero for $i=r$: 
\[b_{rr}^\infty:=\lim_{m\to \infty}b^{(m)}_{rr} \in \CC \setminus \{0\}.\] 
We define the entries $\hat{b}_{ij}^{(m)}$ of a modified sequence $\hat{b}^{(m)}$ for $i+j \le n+1$ as follows. 
\begin{equation}\label{modifiedsp}
\hat{b}_{ij}^{(m)}=\begin{cases} b_{rr}^\infty  & (i,j)=(u_1,u_1) \\ \frac{1}{b_{rr}^\infty}\cdot b^{(m)}_{rr}\cdot b^{(m)}_{u_1u_1} &  (i,j)=(r,r) \\ \frac{1}{b_{rr}^\infty} \cdot b_{rj}^{(m)}\cdot b_{u_1,u_1}^{(m)} & \text{ if } i=r, \t^F_j \ge u_1\text{ and } i+j\le n+1 \\ b_{ij}^{(m)} & \text{otherwise whenever } i+j \le n+1 \end{cases}
\end{equation}
In short, we fix the diagonal entry $b_{u_1u_1}^{(m)}$ to be the nonzero constant $b_{rr}^\infty$ and multiply by $\frac{1}{b_{rr}^\infty} b_{u_1u_1}^{(m)}$ the entries on and above the anti-diagonal in the $r$th row sitting in those columns which cover $u_1$. Recall from \eqref{thetaf} that for $i>l$ $\t^F_i\ge l$. Since by assumption $\t_r^F<u_1\le l$, we must have $r \le l$. Then, 
according to Remark \ref{remark:liealgebra}, there is a unique extension in $\Sp_n$ of these entries to the region below the anti-diagonal, we denote this matrix by $\hat{b}^{(m)}\in \Sp_n$. 
\begin{rem}\label{remark:modified}
For $i+j\le n+1$ the modified entries in \eqref{modifiedsl} are equal to the entries in \eqref{modifiedsp}. In particular, the first $l$ columns of $\tilde{b}^{(m)}$ and $\hat{b}^{(m)}$ are the same.  
\end{rem}



Let $\tilde{p}^\infty=\lim_{m\to \infty} \tilde{b}^{(m)}p_{\tF}$ be the limit point defined by the modified sequence \eqref{modifiedsl}. We show that $\tilde{p}^\infty=\hat{p}^\infty$ and hence $\tilde{\rho}=\hat{\rho}$ on $\calb_\bu^{\Sp,r} \subset \calb_\bu^r$ and part (a) follows from Proposition \ref{firstprop} (a). 

Let $V \in \tF$. Due to Remark \ref{remark:modified} 
\begin{equation}\label{firstl}
\hat{p}_V^\infty=\tilde{p}_V^\infty \text{ holds whenever } \max(V)\le l.
\end{equation}
If $\max(V)= v \ge l+1$ then in fact $V=\{1,\ldots, v\}$ and therefore using the equality $\hat{b}_{ii}^{(m)}=(\hat{b}_{n+1-i,n+1-i}^{(m)})^{-1}$ we get
\begin{equation}\label{hatinfty}
\hat{p}_V^\infty=\lim_{m\to \infty} \prod_{i=1}^v \hat{b}_{ii}^{(m)}\cdot \bigwedge_{i=1}^v e_i=\lim_{m\to \infty} \prod_{i=1}^{n-v} \hat{b}_{ii}^{(m)}\cdot \bigwedge_{i=1}^v e_i=
\hat{p}^\infty_{\{1,\ldots, n-v\}} \wedge \bigwedge_{i=n+1-v}^v e_i.
\end{equation}
Similarly, 
\begin{equation}\label{tildeinfty}
\tilde{p}_V^\infty=\lim_{m\to \infty} \prod_{i=1}^v \tilde{b}_{ii}^{(m)}\cdot \bigwedge_{i=1}^v e_i=\lim_{m\to \infty} \prod_{i=1}^{n-v} \tilde{b}_{ii}^{(m)}\cdot \bigwedge_{i=1}^v e_i=
\tilde{p}^\infty_{\{1,\ldots, n-v\}} \wedge \bigwedge_{i=n+1-v}^v e_i.
\end{equation}
However, by \eqref{firstl} $\tilde{p}^\infty_{\{1,\ldots, n-v\}}=\hat{p}^\infty_{\{1,\ldots, n-v\}}$ and hence $\tilde{p}_V^\infty=\hat{p}_V^\infty$. So $\tilde{p}^\infty=\hat{p}^\infty$ and part (a) is proved.

To prove (b) and (c) we define for $\delta \neq 0$ the modified sequence 
\begin{equation}\label{modifiedspdelta}
\hat{b}_{ij}^{(m),\delta}=\begin{cases} \delta  & (i,j)=(u_1,u_1) \\ \frac{1}{\delta}b^{(m)}_{rr}\cdot b^{(m)}_{u_1u_1} & \text{ if }(i,j)=(r,r) \text{ and } r \le l \\ \frac{1}{\delta}b_{rj}^{(m)}\cdot b_{u_1,u_1}^{(m)} & \text{ if } i=r, j\ge r,  \t^F_j \ge u_1 \text{ and } i+j\le n+1\\ b_{ij}^{(m)} & \text{otherwise whenever } i+j\le n+1\end{cases}
\end{equation}
and its unique extension $\hat{b}^{(m),\delta} \in \Sp_n$. 
\begin{rem} If $r\ge l+1$ then the second and third line in \eqref{modifiedsp} are irrelevant and $\hat{b}^{(m)}$ differs from $b^{(m)}$ only in the $(u_1,u_1)$ and $(n+1-u_1,n+1-u_1)$ diagonal entries. 
\end{rem}
\begin{rem}\label{remark:modified2}
For $i+j\le n+1$ the modified entries in \eqref{modifiedsldelta} are equal to the entries in \eqref{modifiedspdelta}. In particular, the first $l$ columns of $\tilde{b}^{(m),\delta}$ and $\hat{b}^{(m),\delta}$ are the same.  
\end{rem}
If $r\in \bu$ then $\mathrm{VSpec}(\hat{b}^{(m),\delta})=\{u_2,\ldots, u_s\}$. If $r\notin \bu$ and $\t_r^F \ge u_1$ then 
\[\mathrm{VSpec}(\hat{b}^{(m),\delta})=\begin{cases} \{u_2,\ldots, u_s\}& \text{ if } \lim_{m\to \infty}b^{(m)}_{rr}b^{(m)}_{u_1u_1} \neq 0 \\  \{u_2,\ldots, u_s\}\cup \{r\} & \text{ if } \lim_{m\to \infty}b^{(m)}_{rr}b^{(m)}_{u_1u_1} = 0 \end{cases}.\]
Let $\tilde{p}^{\infty,\delta}=\lim_{m\to \infty} \tilde{b}^{(m),\delta}p_{\tF}$ be the limit point defined by the modified sequence \eqref{modifiedsldelta}. The same argument as above shows again that 
\[\hat{p}^{\infty,\delta}=\tilde{p}^{\infty,\delta} \text{ for all } \delta>0\]
and hence by \eqref{tends} we have 
\[\lim_{\delta \to 0} \hat{p}^{\infty,\delta}=\lim_{\delta \to 0} \tilde{p}^{\infty,\delta}=p^\infty.\] 
Therefore 
\[p^{\infty} \in \begin{cases} \overline{\calb^\Sp_{\{u_2,\ldots, u_s\}}} & \text{ if } r\in \bu \\ \overline{\calb^\Sp_{\{u_2,\ldots, u_s\}}} \text{ or } \overline{\calb^\Sp_{\{u_2,\ldots, u_s\} \cup \{r\}}} & \text{ if } r \notin \bu \text{ and } \t^F_r \ge u_1\end{cases} . \]
But $p^\infty \in \calb^\Sp_{\{u_1,\ldots, u_s\}}$ so $p^{\infty}\notin \calb^\Sp_{\{u_2,\ldots, u_s\}}$ and $p^{\infty}\notin \calb^\Sp_{\{u_2,\ldots, u_s\}\cup \{r\}}$ and we are done.   
\end{proof}

Proposition \ref{crucialsymplectic} reduces the proof of Theorem \ref{prop:symplectic} the same way as in the $\SL_n$ case to the simple situation when $\bu=\{u\}$ has a single element. Before we start studying this special case we state the following analog of Lemma \ref{fundamentalcor}.

\begin{lemma}\label{offborelsymplectic}
Let $\a \in R^-$ be a negative root. Let $\calb \subset \overline{B_{\Sp_n} \cdot p_{\tF}}$ be a $B_{\Sp_n}$-invariant subvariety. Assume that for every point $w\in \calb$ there is an element $b_w \in B_{\Sp}$ such that $w$ is fixed by the conjugate $b_w U_\a b_w^{-1}$ of the root subgroup $U_\a \subset \Sp_n$. Then 
\[\dim \overline{\Sp_n \cdot \calb} \le \dim \overline{\Sp_n \cdot p_{\tilde{F}}}-2.\]
\end{lemma}

\begin{proof}
We can simply copy the proof of Lemma \ref{fundamentalcor}. Consider the map 
\[\varphi: \Sp_n \times \calb \to \calw_{\tF},\ \ \ \varphi(g,w) \mapsto g \cdot w.\]
Choose $w\in \calb$ and let $T^w=b_w U_\a b_w^{-1}$ be the corresponding subgroup which fixes $w$. Since $\calb \subset \calw_{\tS}$ is Borel-invariant, the fibre $\varphi^{-1}(g\cdot w)$ contains $(g(bu)^{-1},(bu)\cdot w)$ for $b\in B_{\Sp_n}, u\in b_w U_\a b_w^{-1}$. Since $b_w U_\a b_w^{-1} \cap B_{\Sp_n}=\{1\}$, 
\[\dim(\{bu:b\in B_{\Sp_n},u\in b_w U_\a b_w^{-1}\}=\dim(B_{\Sp_n})+1\] 
and we get   
\begin{multline}\nonumber
\dim(\im(\varphi))=\dim \Sp_n+\dim \calb-\dim(\mathrm{fibre})\le \dim \Sp_n+\dim\overline{B_{\Sp_n} \cdot p_{\tF}}-1-(\dim(B_{\Sp_n})+1)=\\
=\dim \Sp_n/U_S^{\Sp}-2=\dim \overline{\Sp_n\cdot p_{\tilde{F}}}-2.
\end{multline}
\end{proof}

First we study the sets $\calb_u^{\Sp,r}$ with $r\le l$. 

\begin{lemma}\label{firstlemmasymplectic}
Let $\bu=\{u\}$ and $r$ be an integer such that $u <r\le l$. 
Let $\a=\a_{r}-\a_{u}$ so that the corresponding root subgroup has two nonzero off-diagonal entries as in \eqref{rootsubspace} where $x$ sits in the $(r,u)$ and $-x$ in the $(n+1-u,n+1-r)$ entry. Then every point in $\calb^{\Sp,r}_{\bu}$ is fixed by a  conjugate $\hat{A} U_\alpha \hat{A}^{-1}$ for some $\hat{A} \in B_{\Sp_n}$.
\end{lemma} 

\begin{proof} 
Let $p^\infty =\lim_{m\to \infty}b^{(m)}p_{\tilde{Z}} \in \calb^{\Sp,r}_{\bu}$. We define the matrix $A\in B_{\SL_n}$ and the new matrix $\{\tilde{e}_1,\ldots, \tilde{e}_n\}$ satisfying \eqref{pvinftysl2} and \eqref{pvinftysl1} just as in the proof of Lemma \ref{lemma:ufixed}. Since $u<r\le l$, the matrix $A$ has nonzero off-diagonal entries only in the first $l$ column and by Remark \ref{remark:liealgebra}, it has a unique extension $\hat{A}  \in \Sp_n$ whose entries above the anti-diagonal are equal to those entries of $A$. 

We claim that $p^\infty$ is fixed by $\hat{A} U_\a \hat{A}^{-1}$. Equivalently, $p^\infty$ is fixed by $U_\a$ when written in the basis $\{\tilde{e}_1,\ldots, \tilde{e}_n\}$. Since 
$U_\a \subset \Sp_n$ has nonzero off-diagonal entries only at $(r,u)$ and $(n+1-u,n+1-r)$ this follows if we prove that $p^\infty$ is fixed by both $T^{u,\tilde{e}_r}$ and $T^{n+1-r,\tilde{e}_{n+1-u}}$. The first follows from  
\eqref{pvinftysl2} and \eqref{pvinftysl1}. 

To see that $p^\infty$ is fixed by $T^{n+1-r,\tilde{e}_{n+1-u}}$ note that $n+1-r$ and $n+1-u$ are both bigger than $l$ and therefore $T^{n+1-r,\tilde{e}_{n+1-u}}$ fixes $p^\infty_V$ automatically when $\max(V)\le l$. 

When $\max(V)= v \ge l+1$ then in fact $V=\{1,\ldots, v\}$ and therefore using the equality $b_{ii}^{(m)}=b_{n+1-i,n+1-i}^{(m)})^{-1}$ and Remark \ref{remark:modified} we get
\begin{equation}\label{maxvatleastl+1}
p_{\{1,\ldots, v\}}^\infty=\lim_{m\to \infty} \prod_{i=1}^v b_{ii}^{(m)}\cdot \bigwedge_{i=1}^v e_i=\lim_{m\to \infty} \prod_{i=1}^{n-v} b_{ii}^{(m)}\cdot \bigwedge_{i=1}^v e_i
\end{equation}
Now we prove that 
\begin{equation}\label{betweenuandr}
\prod_{i=1}^{t} b_{ii}^{(m)}=0  \text{ whenever } u \le t <r.
\end{equation}
By Proposition \ref{prop:remainstrue} $\lim_{m\to \infty} b_{ii}^{(m)} \in \CC$ exists when $\t_i^F <u$ and since $\t_{u}^F<u$ we have
\[\lim_{m\to \infty} \prod_{i=1}^t b_{ii}^{(m)}=0 \text{ if } u\le t \text{ and } \t_t^F<u.\]
Now let $t<r$ such that $\t_t^F\ge u$ and take $V=\{1,\ldots, t\} \subset \tF$. If $p_V^\infty \neq 0$ then by definition $V\in \calv(p^\infty)^{\Sp}$ and therefore by Proposition \ref{prop:remainstrue}
\[[p_Z^{\infty}] \subset [p_V^\infty]=[e_1\wedge \ldots \wedge e_v]\]
which is a contradiction because $[p_Z^{\infty}]$ has width $r>t$. Therefore $p_V^\infty=0$, that is, 
\[\lim_{m\to \infty} \prod_{i=1}^t b_{ii}^{(m)}=0 \text{ if } t<r  \text{ and } \t_t^F\ge u.\]
Putting these together we get \eqref{betweenuandr}.
Then  
\begin{equation}\label{bigv}
p_{\{1,\ldots, v\}}^\infty=0 \text{ for } n-r<v\le n-u.
\end{equation}
This means that 
\begin{equation}\label{finalconclusion} 
\text{ if } p_{\{1,\ldots, v\}}^\infty \neq 0 \text{ then either: } n+1-r \notin \{1,\ldots, v\} \text{ or: } n+1-r \text{ and } n+1-u \text{ are both in } \{1,\ldots, v\}.
\end{equation}
In both cases $p_{\{1,\ldots, v\}}^\infty$ is fixed by $T^{n+1-r,\tilde{e}_{n+1-u}}$.
\end{proof}

Finally we study the sets $\calb_u^{\Sp,r}$ with $r>l$.

\begin{lemma}\label{secondlemmasymplectic} Let $\bu=\{u\}$ and $r$ be an integer such that $u\le l<r$. 
Let 
\[\a=\begin{cases} \a_{u+1}-\a_{u} & u\le l-1\\ -2\a_l & u=l \end{cases}.\] 
The corresponding negative root subgroup $U_{\a_{u+1}-\a_{u}}$ has nonzero off-diagonal entries as in \eqref{rootsubspace} where $x$ sits at $(u+1,u)$ and $-x$ sits at $(n+1-u,n-u)$, whereas $U_{-2\a_l}$ has $x$ at $(l+1,l)$. Then every point in $\calb^{\Sp,r}_{\bu}$ is fixed by a  conjugate $\hat{A}U_\alpha \hat{A}^{-1}$ for some $\hat{A} \in B_{\Sp_n}$.

\end{lemma}

\begin{proof}  Let $p^\infty \in \calb^{\Sp,r}_{u}$ and $Z\in \calv(p^\infty)^{\Sp}_{\min}$. By Proposition \ref{prop:remainstrue}
\begin{enumerate}
\item $[p^{\infty}_Z] \subset \bigcap_{V\in \calv(p^\infty)^{\Sp}}p_V^\infty$ where  $\calv(p^\infty)^{\Sp}=\{U \in \tF: p^{\infty}_U\neq 0, \t_U \ge u\}$. 
\item $\omega(p^\infty)=r$, and in particular if $\max(V)<r$ then $p_V^\infty \subset \mathrm{Span}(e_1,\ldots, e_r)$ and therefore $[p_Z^\infty]$ cannot sit in $[p_V^{\infty}]$. Therefore by (i) 
\begin{equation}\label{pvinfty1}
\text{ If } \t^F_V \ge u \text{ and } \max(V)<r \text{ then } p_V^\infty =0.
\end{equation}
\end{enumerate}

By Lemma \ref{firstmusymp} 
\[e_j^\infty=\lim_{m\to \infty}b^{(m)}e_j=\mu_{jj}e_j+\ldots +\mu_{1j}e_1 \text{ exists when } \t^F_j<u\]
and in particular, since only $b^{(m)}_{uu}$ tends to zero among the diagonal entries,  we have 
\[\mu_{jj} \neq 0 \text{ when } u<j\le l, \t^F_j <u.\]
Thus we can define the new basis elements $\{\tilde{e}_1,\ldots, \tilde{e}_l\}$ as follows 
\[A: \tilde{e}_j := \begin{cases} e_j^\infty & u<j\le l, \t^F_j <u \\  e_j & \text{ otherwise whenever } j\le l\end{cases} \]
According to Remark \ref{remark:liealgebra}, this can be extended to a linear base change $\hat{A} \in \Sp_n$ to get a new basis $\{\tilde{e}_1,\ldots, \tilde{e}_n\}$. 

If $\t^F_V < u$ then by Lemma \ref{firstmusymp} again $p_V^\infty=\wedge_{v\in V}e_v^\infty=\wedge_{v\in V}\tilde{e}_v$ and $\lim_{m\to \infty} b^{(m)}_{uu}=0$ implies that   
\[e_u^{\infty} \subseteq \mathrm{Span}(e_1,\ldots, e_{u-1})=\mathrm{Span}(\tilde{e}_1,\ldots, \tilde{e}_{u-1})\]
Therefore
\begin{equation}\nonumber
\text{ If } \t^F_V < u \text{ then } [p_V^\infty] \subset \mathrm{Span}(\tilde{e}_1,\ldots, \tilde{e}_{u-1}, \tilde{e}_{u+1}, \ldots, \tilde{e}_n).  .
\end{equation} 
Together with \eqref{pvinfty1} this means that 
\begin{equation}\label{pvinfty2}
\text{ if } p_V^\infty \neq 0 \text{ and } \max(V) < r \text{ then } [p_V^\infty] \subset \mathrm{Span}(\tilde{e}_1,\ldots, \tilde{e}_{u-1}, \tilde{e}_{u+1}, \ldots, \tilde{e}_n).
\end{equation}
If $V\in \tF$ such that $\max(V)=v\ge l+1$ then $V=\{1,\ldots, v\}$ and according to \eqref{bigv} we have $p_{\{1,\ldots, v\}}^\infty=0$ for $n-r<v\le n-u$. Since $u\le l < r$, this means in particular that
\begin{equation}\label{pvinfty3}
p_{\{1,\ldots, v\}}^\infty=0 \text{ for } l+1\le v\le n-u
\end{equation}
Now \eqref{pvinfty2} and \eqref{pvinfty3} implies that 
\[\text{ if } p_V^\infty \neq 0 \text{ then } [p_V^\infty] \subset \mathrm{Span}(\tilde{e}_1,\ldots, \tilde{e}_{u-1}, \tilde{e}_{u+1}, \ldots, \tilde{e}_{n-u-1},\tilde{e}_{n-u+1},\ldots, \tilde{e}_n).\]
Thus $U_{\a_{u+1}-\a_u}$ stabilises $p^\infty$ so $\hat{A}U_{\a_{u+1}-\a_u}\hat{A}^{-1}$ stabilises $p^\infty$ in the old basis.  
\end{proof}

We finished the proof of Theorem \ref{prop:symplectic} and hence Theorem \ref{main2} is proved. 
\end{proof}

\subsection{Borel-regular subgroups of orthogonal groups}\label{subsec:orthogonal}

Let $V$ be an $n$-dimensional complex vector space and $Q:V \times V \to \CC$ a non-degenerate, symmetric bilinear form on $V$. The orthogonal Lie group is then 
\[\SO_{n}(V)=\{A \in \SL_{n}(\CC):Q(Av,Aw)=Q(v,w) \text{ for all } v,w\in V\},\]
and the corresponding symplectic Lie algebra is 
\[\so_{n}(V)=\{A \in \slc_{n}(\CC):Q(Av,w)+Q(v,Aw)=0 \text{ for all } v,w\in V\}.\]
To get a compatible embedding of $\SO_n(\CC)\subset \SL_n(\CC)$ with diagonal maximal torus, we take a basis $\{e_1,\ldots, e_{n}\}$ of $V$ such that $Q$ is given by the matrix $M$ in the form $Q(v,w)=v^tMw$, and we choose $M$ to be the antidiagonal matrix (with $n=2l$ or $n=2l+1$) 
\[M=\left( \begin{array}{ccc} 
 &  & 1  \\
 & \iddots &  \\
1 &  &   \\
\end{array} \right)\]
For $n=2l$ $\SO_n$ has the same maximal torus as $\Sp_n$ and for $n=2l+1$ the maximal torus consist of diagonal matrices $\mathrm{diag}(t_1,\ldots, t_l,1,t_l^{-1},\ldots, t_1)$. The Lie algebra $\so_n$ consists of matrices of the form
\[\left( \begin{array}{cc} 
 A & B \\
 C & -A^{at}\\
\end{array} \right) \text{ for } n=2l \text{ and } \left( \begin{array}{ccc} 
 A & \bv & B \\
 \bw & 0 & -\bv \\
  C & -\bw & -A^{at}\\
\end{array} \right) \text{ for } n=2l+1\]
where $B=-B^{at}, C=-C^{at}$ and $\bv,\bw$ are arrays of length $l$. In particular, the antidiagonal entries are all $0$ in $\so_n$. 
\begin{rem}\label{remark:liealgebraorthogonal}
In particular this means again that any Lie algebra element $A \in \so_n$ is uniquely determined by its entries $\{a_{ij}:i+j< n+1\}$ sitting above the antidiagonal. 
\end{rem}

The characters and cocharacters are the same as in $\Sp_n$ and the Cartan subalgebra $\lieh \subset \so_{n}$ is $l$-dimensional spanned by the diagonal matrices $E_{ii}-E_{l+i,l+i}$ for $1\le i \le l$ whose dual is $\a_i$. The roots are
\begin{equation}\nonumber
\{\pm \a_i \pm \a_j\}_{i<j} \text{ for } n=2l \text{ and } \{\pm \a_i\pm \a_j\}_{i<j} \cup \{\pm \a_i\} \text{ for } n=2l+1
\end{equation}
The positive roots are 
\begin{eqnarray}
n=2l:& R^+=\{\a_i-\a_j\}_{i<j}\cup \{\a_i+\a_j\}_{i<j} \notag \\
n=2l+1:& R^+=\{\a_i-\a_j\}_{i<j}\cup \{\a_i+\a_j\}_{i<j} \cup \{\a_i\} \notag
\end{eqnarray}
The root vectors corresponding to the positive roots have two nonzero entries symmetric about the anti-diagonal as in $\Sp_n$ but here
\[x \text{ sits at }  (i,j) \text{ and } (-x) \text{ sits at } (n+1-j,n+1-i) \text{ if } \alpha=\a_i-\a_j\] 
\[x \text{ sits at } (i,j+l) \text{ and } (-x) \text{ sits at } (l+1-j,n+1-i) \text{ if } \a=\a_i+\a_j, i\neq j\]
\[x \text{ sits at }  (i,l+1) \text{ and } (-x) \text{ sits at } (l+1,n+1-i) \text{ if } \a=\a_i.\]
For a closed subset $S\subset R^+$ let $U_S^{\SO}=\langle U_{\alpha}:\alpha \in S\rangle \subset \SO_{n}$ be the corresponding unipotent subgroup generated by the root subgroups in $\SO_{n}$, normalized by the maximal diagonal torus in $\SO_{n}$.
We define the family $S=\{S_1,\ldots, S_n\}$ of subsets in the same way as for $\Sp_n$, that is, $S_i$ collects the possible non-zero entries in the $j$th column in $U_S^{\Sp} \subset \SO_n \subset \SL_n$:
\[S_j=\{i: \exists u\in U_S^{\SO}\subset \SO_n \subset \SL_n \text{ such that } u_{ij}\neq 0\}\]

\begin{exit}\label{exaorth}	
If $n=4$ and $S=\{\a_1-\a_2,\a_1+\a_2\}$ then the corresponding subgroup $U_S\subset \SO_4(\CC)$ is the maximal unipotent radical of the full upper Borel in $\SO_4(\CC)$, that is
\[U_S=\left\{\left( \begin{array}{cccc} 
1 & a & b & -ab \\
0 & 1 & 0 & -b \\
0 & 0 & 1 & -a \\
0 & 0 & 0  & 1
\end{array} \right):a,b\in \CC\right\}.\] 
Then the nonzero entries of $U_S$ define the sets  
\[S_1=\{1\},S_2=\{1,2\},S_3=\{1,3\},S_4=\{1,2,3,4\}\]
\end{exit}

If $U_S^{\SO} \subset \SO_n$ is Borel-regular then the regular subgroup $U_S^{\SL}$ is Borel regular in $\SL_n$, symmetric about the anti-diagonal. Therefore we can define the crossing point $\g_S$ of $S$ like in the symplectic case. 
\begin{rem}
Note that in $\SO_n$ with $n=2l$ the entry $(l,l+1)$ is always zero and therefore $(l,l+1) \notin S$. This implies that $\g_S\le l-1$. 
\end{rem}
We define the orthogonal fundamental domain $F$ corresponding to a Borel-regular subset $S$ symmetric about the anti-diagonal the same way as in Definition \ref{def:fundamental}. The corresponding point $p_F$ and therefore $p_{\tF}$ has stabiliser $U_S$ in $\SO_n$.

\begin{definition} We define the snipped top right quarter of $\SL_n$ as the domain 
\[\mathbf{Q}=\begin{cases} \{(i,j):1\le i \le l, l+1 \le j \le n\}\setminus \{(l,l+1)\} & \text{ for } n=2l \\ 
\{(i,j):1\le i \le l+1, l+1 \le j \le n\}\setminus \{(l+1,l+1)\} & \text{ for } n=2l+1\end{cases}.\]
We call a Borel regular subgroup $U_S^\SO \subset \SO_n$ \textbf{fat Borel regular} if $\{\a_i+\a_j:1\le i<j\le l\}\subset S$ for $n=2l$ and $\{\a_i+\a_j,\a_i:1\le i<j\le l\}\subset S$ for $n=2l+1$.  
Equivalently, the snipped top right quarter $\mathbf{Q}$ is part of the free parameter domain of the corresponding $U_S^\SL \subset \SL_n$. 
\end{definition}
\begin{figure}[h]
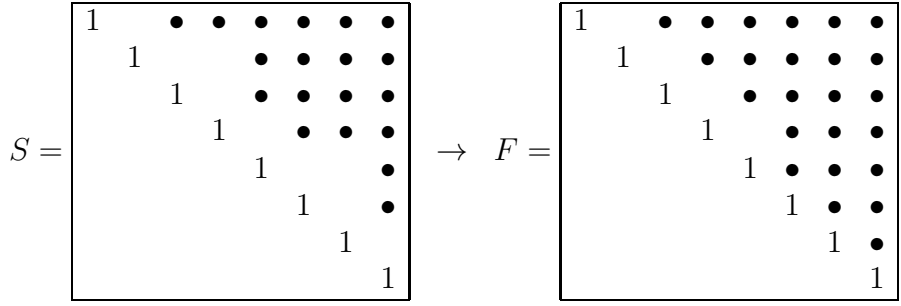

\centering
\begin{tabular}{ccc}
$S=\begin{array}{|cccccccc|}
\hline 
1 &  & \bullet & \bullet & \bullet & \bullet & \bullet & \bullet \\
  & 1 &  &  & \bullet & \bullet & \bullet & \bullet \\
  &   & 1 &    & \bullet & \bullet & \bullet & \bullet \\
  &   &   & 1  &    & \bullet & \bullet & \bullet\\
  &  &   &   & 1 & &  & \bullet\\
  & & & & & 1& & \bullet\\
  & & & & & & 1& \\ 
  & & & & & & &1\\
\hline
\end{array}$ & $\to$ & $F=\begin{array}{|cccccccc|}
\hline 
1 &  & \bullet & \bullet & \bullet & \bullet & \bullet & \bullet \\
  & 1 &  & \bullet  & \bullet & \bullet & \bullet & \bullet \\
  &   & 1 &   & \bullet  & \bullet & \bullet & \bullet \\
  &   &   & 1  &   & \bullet & \bullet & \bullet\\
  &  &   &   & 1 & \bullet & \bullet & \bullet\\
  & & & &  & 1& \bullet & \bullet\\
  & & & & &  & 1& \bullet \\ 
  & & & &  & &  &1\\
\hline
\end{array}$ 
\end{tabular}
\caption{A fat Borel-regular subgroup for $n=8$ and its fundamental domain. Note that $(4,5)$ is missing from the free parameter domain.}
\end{figure}
\begin{figure}[h]
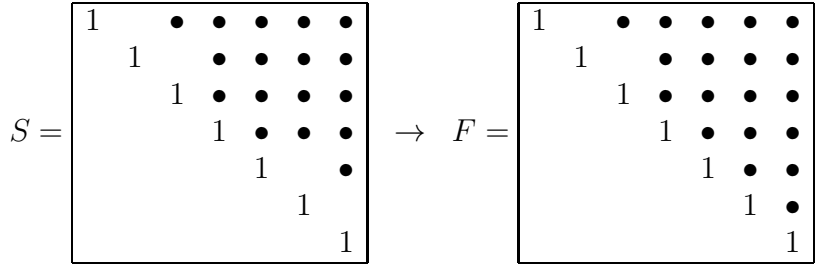

\centering
\begin{tabular}{ccc}
$S=\begin{array}{|ccccccc|}
\hline 
1 &  & \bullet & \bullet & \bullet & \bullet & \bullet  \\
  & 1 &  & \bullet & \bullet & \bullet & \bullet  \\
  &   & 1 &  \bullet  & \bullet & \bullet & \bullet \\
  &   &   & 1  & \bullet   & \bullet & \bullet \\
  &  &   &   & 1 & & \bullet \\
  & & & & & 1& \\
  & & & & & & 1\\ 
\hline
\end{array}$ & $\to$ & $F=\begin{array}{|ccccccc|}
\hline 
1 &  & \bullet & \bullet & \bullet & \bullet & \bullet  \\
  & 1 & & \bullet  & \bullet & \bullet & \bullet  \\
  &   & 1 & \bullet  & \bullet  & \bullet & \bullet \\
  &   &   & 1  & \bullet  & \bullet & \bullet \\
  &  &   &   & 1 & \bullet & \bullet \\
  & & & &  & 1& \bullet \\
  & & & & &  & 1 \\ 
\hline
\end{array}$ 
\end{tabular}
\caption{A fat Borel-regular subgroup for $n=7$ and its fundamental domain.}
\end{figure}

\begin{theorem}\label{prop:orthogonal}
Let $n=2l$ or $n=2l+1$ and $F=\{F_1,\ldots, F_n\}$ be the orthogonal fundamental domain corresponding to a fat Borel-regular subgroup $U_S \subset \SO_n$. Then the pair $(\calw_{\tF},p_{\tF})$ is a Grosshans pair for $U_F$. This proves Theorem \ref{main2} for orthogonal groups. 
\end{theorem}
  
\begin{proof}
First we assume $n=2l$. The key observation is the following stronger version of Lemma \ref{lemma:u1}.
\begin{lemma}\label{lemma:u1orthogonal} Let $n=2l$. If $\bu=\{u_1< \ldots <u_s\}$ is such that $u_1\ge l$ then $\calb^\SO_\bu=\emptyset$.
\end{lemma}
\begin{proof} The $u_1>l$ case is the same as in Lemma \ref{lemma:u1}. Assume $u_1=l$. Then 
\[\lim_{m\to \infty} b^{(m)}_{l+1l+1}=\lim_{m\to \infty} (b^{(m)}_{ll})^{-1}=\infty \text{ and } \lim_{m\to \infty} b^{(m)}_{ii}\in \CC \setminus \{0\} \text{ for } 1\le i <l.\]
Since $S$ is flat, $F_{l+1}=\{1,\ldots, l-1,l+1\}$ and the coefficient of $e_1\wedge \ldots \wedge e_{l-1} \wedge e_{l+1}$ in $p^{\infty}_{F_{l+1}}$ is  
\[p_{S_{l+1}}^\infty[e_1\wedge \ldots \wedge e_{l-1} \wedge e_{l+1}]=\lim_{m\to \infty} b^{(m)}_{l+1l+1} \cdot   \prod_{i=1}^{l-1} b^{(m)}_{ii}=\infty \]
a contradiction.
\end{proof}
The proof of Theorem \ref{prop:symplectic} applies with two minor changes for the proof of Theorem \ref{prop:orthogonal}. The only difference we have to keep in mind is that for $\SO_{2l}$ $F^{\SO}_{l+1}=\{1,\ldots, l-1,l+1\}$ whereas in $Sp_{2l}$ it was $F^{\Sp}_{l+1}=\{1,\ldots, l+1\}$. This means that $F^{\SO}=F^{\Sp} \cup F^{\SO}_{l+1}$ and this extra set results minor changes in the proof of Lemma \ref{firstlemmasymplectic} and Lemma \ref{secondlemmasymplectic} as follow.
\begin{itemize}
\item In the proof of Lemma \ref{firstlemmasymplectic} the first part proving that $T^{u,e_r}$ fixes $p^\infty$ remains the same. To prove that $T^{n+1-r,e_{n+1-u}}$ fixes $p^\infty$ we only need to worry about those $p_V^\infty$'s where $\max(V)\ge l+1$. In order to prove \eqref{finalconclusion} we distinguish two cases: 

(a) If $r\le l-1$ then $n+1-r\ge l+2$. However, for $\max(V)=v\ge l+2$ we have $V=\{1,\ldots, v\}$ so \eqref{maxvatleastl+1} holds and therefore \eqref{betweenuandr} implies that $p_{\{1,\ldots, v\}}^\infty =0$ for $n+1-r \le v \le n-u$ again.

(b) If $r=l$ then either $p_{\{1,\ldots, l-1,l+1\}}^\infty=0$ and the extra subset $F^{\SO}_{l+1}$ added to $F^{\Sp}$ does not affect the proof, or $p^\infty_{\{1,\ldots, l-1,l+1\}}\neq 0$ but then 
\[e_r\in [p_Z^\infty] \subset [p^\infty_{F_{l+1}}]=[p_{1,\ldots, l-1,l+2}^\infty]\]
and the only way this can happen is that $[p_{1,\ldots, l-1,l+2}^\infty]=\mathrm{Span}(e_1,\ldots, e_l)$. But then $e_{l+1}=e_{n+1-r}$ is not contained in the only problematic set $[p_{\{1,\ldots, l-1,l+1\}}]^\infty$ and the proof of the symplectic case works here again. 

\item The second case in Lemma \ref{secondlemmasymplectic} does not
make sense in the orthogonal case: $-2\a_l$ is not a root for $\SO_{2l}$. But Lemma \ref{lemma:u1orthogonal} tells us that $u=l$ can not happen and in fact $u\le l-1$ ensures that the extra set $\{1,\ldots, l-1,l+1\}$ which we added to $F^{\Sp_{2l}}$ does not affect the proof. Indeed, this is clear when $u\le l-2$ because in this case $n-u\ge l+2$ and hence to prove that $T^{n-u,e_{n+1-u}}$ fixes $p^\infty$ it is enough to have the following weaker version of \eqref{pvinfty3}:
\[p_{\{1,\ldots, v\}}^\infty=0 \text{ for } l+2 \le v \le n-u.\] 
But subsets with $\max(V)\ge l+2$ are the same in $F^\Sp$ and $F^\SO$ and so this follows exactly the same way as \eqref{pvinfty3}. 

If $u=l-1$ and $r\ge l+2$ then $p_{\{1,\ldots, l-1,l+1\}}^\infty=0$ and $F_{l+1}^{\SO}$ does not make any difference.

Finally, if $u=l-1$ is the only element of the vanishing spectrum and $r=l+1$ then $\lim_{m\to \infty}b_{ll}^{(m)}\neq 0$ and therefore $\lim_{m\to \infty}|b_{l+1l+1}^{(m)}|=\lim_{m\to \infty}|(b_{ll}^{(m)})|^{-1} <\infty$. But then the coefficient of $e_1\wedge \ldots \wedge e_{l-1}\wedge e_{l+1}$ in $p_{\{1,\ldots, l-1,l+1\}}^\infty$ is 
\[\lim_{m\to \infty}(b_{11}^{(m)} \cdots b_{l-1l-1}^{(m)}) \cdot \lim_{m\to \infty}b^{(m)}_{l+1l+1}=0\]
because the first limit is $0$ (the $l-1$th term tends to $0$, the rest to some nonzero constant) and the second limit is finite. Therefore 
\[e_r \in [p_Z^\infty] \subset [p_{\{1,\ldots, l-1,l+1\}}^\infty] \subset [\mathrm{Span}(e_1,\ldots, e_{r-1})],\]
a contradiction.
\end{itemize}

Now assume $n=2l+1$. Since the diagonal entry $b_{l+1,l+1}^{(m)}$ is constant $1$ in $\SO_n$, $\bu$ can not contain $l+1$ and Lemma \ref{lemma:u1} holds without change. We furthermore add the following observation.
\begin{lemma}
$\calb^{\SO,l+1}_\bu=\emptyset$ for arbitrary $\bu=\{u_1<\ldots <u_s\}$. 
\end{lemma} 

\begin{proof}
Assume $p^\infty=\lim_{m\to \infty}b^{(m)}p_{\tF} \in \calb^{\SO,l+1}_\bu$. By definition there is a $Z\in \tF$ with $\t^F_Z \ge u_1$ such that $p_Z^\infty \neq 0$ and 
\begin{enumerate}
\item $[p^{\infty}_Z] \subset \bigcap_{\substack{V\in \tF, p^{\infty}_{V}\neq 0 \\ \t_V \ge u_1}}[p^{\infty}_{V}]$
\item $\omega([p_Z^\infty])=l+1$. For this to happen $\max(Z)=z \ge l+1$ must hold, and therefore $Z=\{1,\ldots, z\}$. But if $\omega([p_Z^\infty])=l+1$ then $[p_Z^\infty]=[p_{\{1,\ldots, z\}}^\infty]=[e_1 \wedge \ldots \wedge e_z]$ is a subspace of $\mathrm{Span}(e_1,\ldots, e_{l+1})$ by definition which means that $z=l+1$ and $Z=\{1,\ldots, l+1\}$. Then 
\[0\neq p_Z^\infty=\lim_{m\to \infty} (\prod_{i=1}^{l+1} b_{ii}^{(m)}) \cdot \bigwedge_{i=1}^{l+1}e_i\]
But $b^{(m)}\in \SO_{2l+1}$ and therefore the diagonal entry $b_{l+1l+1}^{(m)}=1$ for all $m$. Hence
\[p_{\{1,\ldots, l\}}^\infty=\lim_{m\to \infty} (\prod_{i=1}^{l} b_{ii}^{(m)}) \cdot \bigwedge_{i=1}^{l}e_i\neq 0,\]
which means that $\{1,\ldots, l\}$ is a minimal subset for $p^\infty$ contradicting to the minimality of $Z$ with respect to $\preceq$ (see Definition \ref{def:bur}) because $\max(\{1,\ldots, l\}=l<\max(Z)=l+1$.

\end{enumerate}
\end{proof}

In particular, this means that either $r\le l$ or $r \ge l+2$ and the proof of Theorem \ref{prop:symplectic} applies again without change, including Lemma \ref{firstlemmasymplectic} and Lemma \ref{secondlemmasymplectic}.
\end{proof}

\section{A partial result for general regular subgroups of $\SL_n$}\label{strategy}

This section gives partial affirmative answer to the Popov-Pommerening conjecture for general regular subgroups of $\SL_n$ corresponding to arbitrary closed family $S\subset R^+$. We prove Theorem \ref{mainc}. Let $G$ be a connected, simply connected, simple linear algebraic group over the algebraically closed subfield $k$ of $\CC$, and $U_S\subset G$ a unipotent subgroup normalized by a maximal torus $T$ of $G$ corresponding to the closed subset $S \subset R^+$, where $U_S$ is not necessarily block regular.  

\begin{definition} 
We call $G(\overline{T\cdot p_{\tS}}) \subset \overline{G\cdot p_{\tS}} \subset \calw_{\tS}$ the toric closure of $G\cdot p_{\tS}$. Points and components of $G(\overline{T\cdot p_{\tS}})\setminus G\cdot p_{\tS}$ are called toric boundary points and components. 
\end{definition}
 We are ready to prove Theorem \ref{mainc} on toric boundary components. Unfortunately we cannot prove the same for non-toric boundary components, that is, components of $\overline{G\cdot p_{\tS}} \setminus G(\overline{T\cdot p_{\tS}})$. 

\begin{proof}[Proof of Theorem \ref{mainc}] Let $T \subset \SL_n$ be the diagonal torus.  
Points of $\overline{T\cdot p_{\tS}}$ are limits of the form
\[p^\infty=\lim_{m\to \infty} 
\left( \begin{array}{cccc}
b_{11}^{(m)} & 0 & \cdots & 0\\
0 & b_{22}^{(m)} &   & \\
\vdots & & \ddots & 0 \\
0 & & 0 & b_{nn}^{(m)}
\end{array}\right)\cdot p_{\tS} =\bigoplus_{V\in \tS} p^\infty_V\]
where $p^\infty_V=\lim_{m\to \infty} \wedge_{i \in V} b_{ii}^{(m)}e_i$.
According to Lemma \ref{borelorbit} if $p^\infty$ is a boundary point, that is $p^\infty \in \overline{T\cdot p_{\tS}} \setminus T\cdot p_{\tS}$, then there is a smallest index $1\le s \le n$ such that 
\[\lim_{m\to \infty}b_{ss}^{(m)}=0.\]

Define
\[t=\min\{j:\exists V \in \tS \text{ such that } s\in V, j=\max(V) \text{ and } p^{\infty}_V \neq 0\}.\]
Note that $t$ is well-defined because $V=\{1,\ldots, n\}\in \tS$ and $\lim_{m\to \infty}\Pi_{i=1}^n b_{ii}^{(m)}=1$, so the defining set above is nonempty. Furthermore $t>s$ holds by the minimality of $s$. We call $(s,t)$ the type of $p^\infty$. 

\begin{rem}
If the vanishing spectrum of $(b^{(m)})$ is $\bu=\{u_1<\ldots u_s\}$ then $s=u_1$ and $s$ is uniquely determined by $p^\infty$ according to Remark \ref{u1determined}. Moreover, $t$ plays the role of the width of $p^\infty$ and it is again determined by $p^\infty$. 
\end{rem} 

Let $Z\in \tS$ be one of the minimising subsets in the definition of $t$, that is  
\begin{equation}\nonumber
s \in Z, t=\max(Z) \text{ and } p^\infty_{Z}=\lim_{m\to \infty}\Pi_{i \in Z}b_{ii}^{(m)} \wedge_{i\in Z}e_i\neq 0.
\end{equation}
We prove that $p^{\infty}$ is $(s,e_t)$-fixed (see Definition \ref{ivfixed}). Assume there is a $V \in \tS$ such that 
\begin{equation}\label{crucial2}
 p^\infty_V=\lim_{m\to \infty} (\prod_{i\in V}b_{ii}^{(m)})\cdot \wedge_{i\in V} e_i\neq 0, s \in V  \text{ but } t \notin V.
\end{equation}
Now $V \cup Z,V \cap Z \in \tS$ and 
\[\lim_{m\to \infty} \Pi_{i\in V \cup Z}b_{ii}^{(m)}=
\lim_{m\to \infty} \frac{\Pi_{i\in V}b_{ii}^{(m)}\Pi_{i\in Z}b_{ii}^{(m)}}{\Pi_{i\in V\cap Z}b_{ii}^{(m)}}.\]
The limit of the numerator is finite and nonzero from the definition of $V$ and $Z$.
But $s \in V\cap Z$ and $t\notin V$ so $\max(V\cap Z)<t$ and therefore by the definition of $t$ the limit of the denominator is $0$. This is a contradiction as the left hand side is the coefficient of $\wedge_{i\in V\cup Z}e_i$ in $p^{\infty}_{V \cup Z}$. 
So there is no $V\in \tS$ satisfying \eqref{crucial2} which means that $p^\infty$ is fixed by $T^{s,e_t}(\l) \in \SL_n(\CC)$ and therefore $p^\infty$ is $(s,e_t)$-fixed.
For $1\le s<t \le n$ let 
\[\calb_{s,t}=\{p^\infty \in \overline{T\cdot p_{\tS}} \setminus T\cdot p_{\tS}: \text{ the type of } p^\infty \text{ is } (s,t)\}.\] 
Then 
\[\SL_n \cdot \overline{T\cdot p_{\tS}} \setminus T\cdot p_{\tS}=\bigcup_{1\le s<t\le n} \SL_n \cdot \calb_{s,t}.\]
We adapt the proof of Lemma \ref{fundamentalcor} to show that 
\[\dim(\SL_n \cdot \calb_{s,t}) \le \dim(\SL_n \cdot p_{\tS})-2,\]
which implies Theorem \ref{mainc}.
Let's start with the observation that $\calb_{s,t}$ is $T$-invariant and also $U_S$-invariant for any $1\le s<t\le n$.  This latter follows from the fact that $U_S$ is normalized by $T$ and fixes $p_{\tS}$ and therefore $U_S$ fixes each point in $T\cdot p_{\tS}$ and, then, each point in $\calb_{s,t}$.
Consider the map 
\[\varphi: \SL_n(\CC) \times \calb_{s,t} \to \calw_{\tS},\ \ \ \varphi(g,w) \mapsto g \cdot w.\]
Let $w\in \calb_{s,t}$. Since $\calb_{s,t} \subset \calw_{\tS}$ is $U_S \rtimes T$-invariant, the fibre $\varphi^{-1}(g\cdot w)$ contains 
\[(g(hT^{s,e_t}(\lambda))^{-1},(hT^{s,e_t}(\lambda))\cdot w)\]
for $h\in U_S \rtimes T, \lambda\in \CC$. Since $\{T^{s,e_t}(\l):\l \in \CC\} \cap U_S \rtimes T=\{1\}$, 
\[\dim(\{hT^{s,e_t}(\lambda):h\in U_S \rtimes T,\l \in \CC\}=\dim(U_S)+n+1\] 
and we get    
\begin{multline}\nonumber
\dim(\im(\varphi))=\dim(\SL_n)+\dim(\calb_{s,t})-\dim(\mathrm{fibre})\le \dim(\SL_n)+\dim(\overline{T \cdot p_{\tS}})-1-(\dim(U_S)+n+1)=\\
=\dim \SL_n(\CC)/U_S -2=\dim \overline{\SL_n(\CC)\cdot p_{\tilde{S}}}-2.
\end{multline}

\end{proof}
\section{A remark on configuration varieties and Bott-Samelson varieties}
Configuration varieties are a powerful tool in representation theory and geometry of the reductive group $G$. If $B\subset G$ is a Borel subgroup then these varieties are certain subvarieties in the product of flag varieties $(G/B)^l$. In \cite{magyar} Magyar describes them as closures of $B$-orbits in $(G/B)^l$, which is relevant to our construction, and therefore we give a short summary in the special case when $G=\SL_n(\CC)$, keeping \cite{magyar} as the leading reference.  


Let $B_n \subset \SL_n(\CC)$ denote the Borel of upper triangular matrices. Define a {\em subset family} to be a collection $D=\{C_1,\ldots, C_m\}$ of subsets $C_k \subset [n]=\{1,\ldots, n\}$. The order is irrevelant in the family, and we do not allow repetitions. 
Let $\CC^n$ have the standard basis $\{e_1,\ldots, e_n\}$ and for any subset $C\subset [n]$ define the subspace
\[\calq^C=\mathrm{Span}_\CC\{e_j:j\in C\}\in \grass(|C|,n).\] 
This point is fixed by the diagonal torus $T \subset \SL_n(\CC)$, and so we can associate a $T$-fixed point to the subset family in the product of Grassmannaians:
\[z_D=(\calq^{C_1},\ldots, \calq^{C_m})\in \grass(D)=\grass(|C_1|,n)\times \ldots \times \grass(|C_m|,n)\] 
The {\em configuration variety} of $D$ is the closure of the $\SL_n(\CC)$-orbit of $z_D$:
\[\cala_D=\overline{\SL_n(\CC) \cdot z_D}\subset \grass(D),\]
and the {\em flagged configuration variety} is the closure of the Borel orbit:
\[\cala_D^{B}=\overline{B_n \cdot z_D}\subset \grass(D).\]
There is an important class of subset families associated to subsets of the Weyl group $W$ of the reductive group. In the case of $\SL_n(\CC)$ to a list of permutations
$\bw=(w_1,\ldots,w_l)$, $w_k\in W$, and a list of indices $\bj=(j_1,\ldots, j_l),1\le j_k\le n$, 
we associate a subset family:
\[D=D_{\bw,\bj}=\{w_1[j_1],\ldots,w_l[j_l]\},\]
where $w[j] =\{w(1),w(2),\ldots,w(j)\}$.
Now suppose the list of indices $\bi=(i_1,i_2,\ldots,i_l)$ encodes a reduced decomposition $w=s_{i_1}s_{i_2}\ldots s_{i_l}$ of a permutation into a minimal number of simple
transpositions. Let $w=(s_{i_1},s_{i_1}s_{i_2},\ldots ,w)$ and define the {\em reduced chamber family} $D_{\bi}:=D_{\bw,\bi}$. The {\em full chamber family} is 
\[D_{\bi}^+=\{[1],[2],\ldots,[n]\}\cup D_{\bi}.\]
A subfamily $D \subset D_\bi^+$ is called a {\em chamber subfamily}. Le-Clerk and Zelevinsky in \cite{clerkzelevinsky} gave a characterization of these as follows. For two sets $S_1,S_2 \subset [n]$ we say $S_1$ is elementwise less than $S_2$, $S_1<^e S_2$, if
$s_1<s_2$ for all $s_1\in S_1,s_2\in S_2$. Now, a pair of subsets $C_1,C_2\subset [n]$ is strongly
separtated if
$(C_1 \setminus C_2)<^e (C_2\setminus C_1)$ or 
$(C_2 \setminus C_1)<^e (C_1\setminus C_2)$
holds. A family of subsets is called {\em strongly separated} if each pair of subsets in it is strongly separated. Le Clerk and Zelevinsky proved that a subset family $D$ is a chamber sufamily, $D \subset D_{\bi}$ for some $\bi$ if and only if it is strongly separated. 

If $D=D_\bi$ is a chamber family, then the corresponding flagged configuration variety $\cala_{D}^B$ is called {\em Bott-Samelson variety}. 

Very little is known about general configuration varieties. They can be badly singular, however, certain of them are well understood because they can be desingularized by Bott-Samelson varieties, which are always smooth. 

The link to our construction is straightforward; if $\cals=\{S_1,\ldots, S_n\}$ denotes the subset family formed from the columns of the star pattern $S$ corresponding the regular subgroup $U_S\subset \SL_n(\CC)$, then there is a natural map
\[\pi_S:(\SL_n(\CC)\cdot p_{\tS}) \to \cala_{\cals} \text{ where } A\cdot p_{\tS}  \mapsto A \cdot z_{\cals} \text{ for } A\in \SL_n(\CC).\]
This  map does not extend to the closure. In short, our space $\overline{\SL_n(\CC) \cdot p_{\tS}}$ is a weighted affine configuration space where the weights are different tensor powers of $\te_n$. 
Unfortunately, the subset family $\cals=\{S_1,\ldots, S_n\}$ is not necessarily strongly separated and therefore not a chamber subfamily in general. This leaves the question of desingularisation of $\overline{\SL_n(\CC) \cdot p_{\tS}}$ open.

\bibliographystyle{amsalpha}
\bibliography{biblio}

\end{document}